\documentclass[a4paper, 11pt]{amsart}

\usepackage{amsmath, amssymb, amsthm, color, tikz, tikz-cd, hyperref, mathtools, stmaryrd, url, enumerate, mathrsfs}
\usetikzlibrary{calc}
\usepackage[capitalize]{cleveref}
\crefformat{ass}{Assumption #2$\dagger$#3}
\crefformat{step}{Step #2#1#3}
\crefformat{con}{Conjecture #2$\diamondsuit$#3}
\crefname{prop}{Proposition}{Propositions}
\crefname{rmkx}{Remark}{Remarks}

\let\originalleft\left
\let\originalright\right
\renewcommand{\left}{\mathopen{}\mathclose\bgroup\originalleft}
\renewcommand{\right}{\aftergroup\egroup\originalright}

\newcommand{\lb}{\left(}
\newcommand{\rb}{\right)}
\newcommand{\diff}{\mathrm{d}}
\newcommand{\eps}{\varepsilon}
\newcommand{\Q}{\mathbb{Q}}
\newcommand{\R}{\mathbb{R}}
\newcommand{\C}{\mathbb{C}}
\newcommand{\F}{\mathbb{F}}
\newcommand{\Z}{\mathbb{Z}}
\renewcommand{\P}{\mathbb{P}}
\newcommand{\pd}{\partial}
\renewcommand{\phi}{\varphi}
\newcommand{\ip}[2]{\left\langle #1, #2 \right\rangle}
\let\oldref\ref
\renewcommand{\eqref}[1]{\textup{(\oldref{#1})}}
\newcommand{\CO}{\operatorname{\mathcal{CO}}}

\newcommand{\Spec}{\operatorname{Spec}}
\newcommand{\Proj}{\operatorname{Proj}}
\newcommand{\Hom}{\operatorname{Hom}}
\newcommand{\Jac}{\operatorname{Jac}}
\newcommand{\HBM}{H^{\mathrm{BM}}}
\newcommand{\QHB}{QH_{\mathrm{B}}}
\newcommand{\Tor}{\operatorname{Tor}}
\newcommand{\m}{\mathfrak{m}}
\newcommand{\ks}{\mathfrak{ks}}
\newcommand{\qcup}{\mathbin{*}}

\newcommand{\hv}{\widehat{z}} 
\newcommand{\lv}{z} 
\newcommand{\llangle}{\langle \kern-2pt \langle}
\newcommand{\rrangle}{\rangle \kern-2pt \rangle}
\newcommand{\clos}{\operatorname{clos}}
\newcommand{\lspan}[1]{\langle #1 \rangle}
\newcommand{\LHD}{\Lambda_0 \llangle H \rrangle_\Delta}

\newcommand{\sphere}[2]
{
\def\u{-0.05}
\def\v{0.32}
\begin{scope}[shift={#1}]
\draw[domain=-1:181, samples=100, smooth] plot ({#2*cos(atan(\u)+\x)}, {#2*sin(atan(\u)+\x)}, {0});
\draw[domain=0:359, samples=200, fill=white, fill opacity=0.65, smooth cycle] plot ({cos(\x)*#2/sqrt((1+\u^2)*cos(\x)^2+2*\u*\v*cos(\x)*sin(\x)+(1+\v^2)*sin(\x)^2)}, {(\u*cos(\x)+\v*sin(\x))*#2/sqrt((1+\u^2)*cos(\x)^2+2*\u*\v*cos(\x)*sin(\x)+(1+\v^2)*sin(\x)^2)}, {sin(\x)*#2/sqrt((1+\u^2)*cos(\x)^2+2*\u*\v*cos(\x)*sin(\x)+(1+\v^2)*sin(\x)^2)});
\draw[domain=180:360, samples=100, smooth] plot ({#2*cos(atan(\u)+\x)}, {#2*sin(atan(\u)+\x)}, {0});

\end{scope}
} 

\theoremstyle{plain}
\newtheorem{thm}{Theorem}[section]
\newtheorem{lem}[thm]{Lemma}
\newtheorem{prop}[thm]{Proposition}
\newtheorem{cor}[thm]{Corollary}

\newtheorem{mthm}{Theorem}
\newtheorem{mcor}[mthm]{Corollary}

\theoremstyle{remark}
\newtheorem{rmkx}[thm]{Remark}
\newenvironment{rmk}
  {\pushQED{\qed}\rmkx}
  {\popQED\endrmkx}
\newtheorem{exx}[thm]{Example}
\newenvironment{ex}
  {\pushQED{\qed}\exx}
  {\popQED\endexx}
  
  \theoremstyle{definition}
\newtheorem{defnx}[thm]{Definition}
\newenvironment{defn}
  {\pushQED{\qed}\defnx}
  {\popQED\enddefnx}

\title{Quantum cohomology and closed--string mirror symmetry for toric varieties}
\author{Jack Smith}
\address{Department of Mathematics\\University College London\\Gower Street\\London\\WC1E 6BT}
\email{jack.smith@ucl.ac.uk}

\begin{document}

\begin{abstract}
We give a short new computation of the quantum cohomology of an arbitrary smooth toric variety $X$, by showing directly that the Kodaira--Spencer map of Fukaya--Oh--Ohta--Ono defines an isomorphism onto a suitable Jacobian ring.  The proof is based on the purely algebraic fact that a class of \emph{generalised Jacobian rings} associated to $X$ are free as modules over the Novikov ring.  In contrast to previous results of this kind, $X$ need not be compact.  When $X$ is monotone the presentation we obtain is completely explicit, using only well-known computations with the standard complex structure.
\end{abstract}

\maketitle

\section{Introduction and main results}

\subsection{Overview}
\label{sscOverview}

Toric varieties provide an important testing ground in algebraic and symplectic geometry, carrying enough structure to render many computations tractable but still being general enough to include plenty of non-trivial examples.  They also have the useful feature that their geometric properties can often be translated into combinatorial properties of the defining fans, or dually---and more relevant to us since they encode the symplectic structure---moment polyhedra.

For example, let $X$ be a toric variety of complex dimension $n$, defined by a moment polyhedron
\begin{equation}
\label{Polytope}
\Delta = \{ x \in \mathfrak{t}^* : \ip{x}{\nu_j} \geq -\lambda_j \text{ for } j = 1, \dots, N\}
\end{equation}
as described in \cref{sscToricBG}.  Letting $F_1, \dots, F_N$ denote the facets (codimension-$1$ faces) of $\Delta$, namely
\[
F_j = \Delta \cap \{x : \ip{x}{\nu_j} = -\lambda_j\},
\]
the \emph{Stanley--Reisner ring} of $X$ is defined to be
\[
SR(X) = \C[Z_1, \dots, Z_N] / ( Z_{j_1} \dots Z_{j_k} : F_{j_1} \cap \dots \cap F_{j_k} = \emptyset )
\]
(traditionally a Stanley--Reisner ring is associated to a \emph{simplicial complex} and here the relevant one is $\pd \Delta^*$, defined in the proof of \cref{SRisCM}).  The classical cohomology ring of $X$ then has the explicit \emph{Stanley--Reisner presentation}
\begin{equation}
\label{SRPresentation}
H^*(X; \C) \cong SR(X) \Big/ \Big(\sum_{j=1}^N \nu_j Z_j\Big),
\end{equation}
where the class $H_j$, Poincar\'e dual to the toric divisor $D_j$ given by the moment map preimage of $F_j$, is sent to the variable $Z_j$.

\begin{rmk}
We shall assume throughout that $\Delta$ has a vertex, which is equivalent to $X$ not admitting a splitting $X = \C^* \times X'$ (see \cref{VertexRmk}).  We otherwise allow $\Delta$, and hence $X$, to be non-compact, and this possible non-compactness is why we use the term `polyhedron' rather than `polytope'.
\end{rmk}

The goal of the present paper is to give a similar description of the \emph{quantum cohomology} of $X$.  Recall that this is a deformation of the classical cohomology over the \emph{universal Novikov ring}
\[
\Lambda_0 = \Big\{\sum_{j=1}^\infty a_jT^{l_j} : a_j \in \C \text{ and } l_j \in \R_{\geq 0} \text{ with } l_j \rightarrow \infty \text{ as } j \rightarrow \infty\Big\},
\]
in which the product is modified by counts of rational curves, weighted by $T^\text{area}$.  In particular, there is a natural identification of $\Lambda_0$-modules
\[
QH^*(X; \Lambda_0) \cong H^*(X; \Lambda_0) \cong H^*(X; \C) \otimes_\C \Lambda_0,
\]
and after reducing modulo the maximal ideal $\Lambda_+ \subset \Lambda_0$ (in which only strictly positive powers of $T$ are allowed) this becomes an isomorphism of algebras over $\Lambda_0 / \Lambda_+ = \C$.  For us, all rings and algebras, and homomorphisms between them, are implicitly unital.

Our starting point is to define a $T$-adically complete $\Lambda_0$-algebra $\LHD$ associated to $\Delta$, which comes with certain distinguished elements $\lv_1, \dots, \lv_N$.  The reduction of this algebra modulo $\Lambda_+$ is naturally identified with the Stanley--Reisner ring (\cref{SRRing}), and under this identification $\lv_j$ corresponds to $Z_j$.  Our main theorem is the following algebraic result:

\begin{mthm}
\label{Theorem1}
For any elements $\hv_1, \dots, \hv_N$ of $\LHD$ satisfying
\[
\hv_j = \lv_j \mod \Lambda_+
\]
for all $j$, the $\Lambda_0$-algebra
\begin{equation}
\label{eqGenJac}
\LHD \Big/ \clos\Big(\sum_{j=1}^N \nu_j \hv_j\Big)
\end{equation}
is free as a $\Lambda_0$-module, where $\clos$ denotes closure in the $T$-adic topology.
\end{mthm}

The main ingredient in the proof of \cref{Theorem1} is \cref{RegSeq}, which asserts that the linear expressions $\sum_j \nu_j Z_j$ form a regular sequence in $SR(X)$.  This in turn is proved using a dimension count, after showing that $SR(X)$ is Cohen--Macaulay (\cref{SRisCM}).

\cref{Theorem1} has the following consequence:

\begin{mcor}
\label{corModIso}
In the setup of \cref{Theorem1}, any $\Lambda_0$-module homomorphism
\[
QH^*(X; \Lambda_0) \rightarrow \LHD \Big/ \clos\Big(\sum_{j=1}^N \nu_j \hv_j\Big)
\]
which sends $H_j$ to $\hv_j$ for all $j$ is an isomorphism.
\end{mcor}

\begin{proof}
Both sides are free $\Lambda_0$-modules, so it is enough to show that the reduction modulo $\Lambda_+$ is an isomorphism, and this follows from the Stanley--Reisner presentation \eqref{SRPresentation} of $H^*(X; \C)$.
\end{proof}

Of course, \cref{corModIso} is only useful if we can find a module homomorphism with the required properties, and it is only interesting if the homomorphism is actually a map of \emph{algebras} (since the interesting feature of $QH^*(X; \Lambda_0)$ is its product).  Fortunately, there is a natural geometric construction of such an algebra homomorphism, via a form of closed--open string map:

\begin{prop}[{Fukaya--Oh--Ohta--Ono \cite{FOOOMSToric}, adapted to our setting in \cref{secks}}]
\label{propFOOOks}
There exists a \emph{superpotential} $W$ in $\LHD$, which decomposes naturally as a sum
\[
W = W_1 + \dots + W_N.
\]
Each $W_j$ satisfies
\[
W_j = \lv_j \mod \Lambda_+,
\]
and there is a $\Lambda_0$-algebra homomorphism
\[
\ks : QH^*(X; \Lambda_0) \rightarrow \LHD \Big/ \Big(\sum_{j=1}^N \nu_j W_j \Big)
\]
(the \emph{Kodaira--Spencer map}) which sends $H_j$ to $W_j$.
\end{prop}

In \cite{FOOOMSToric} only \emph{compact} toric varieties are explicitly considered, but compactness plays no role in the proof of \cref{propFOOOks}.  We shall see in \cref{DjGenerate} that the toric divisor classes $H_j$ generate $QH^*(X; \Lambda_0)$ as a $\Lambda_0$-algebra (this uses the fact that $\Delta$ has a vertex), so $\ks$ is in fact uniquely determined by its values $W_j$ on these classes.  In \cite{SmithGeneration} we interpret $\ks$ in terms of the Hochschild cohomology of certain toric fibres which split-generate the Fukaya category of $X$ when it is compact.

\begin{rmk}
\label{rmkGroundField}
\cref{Theorem1} holds over any ground field, not just $\C$, but in order to define quantum cohomology one is generally restricted to working in characteristic zero since the moduli spaces of curves being counted may only admit virtual fundamental chains over $\Q$, not $\Z$.  Moreover, \cref{propFOOOks} relies on the technical framework of Fukaya--Oh--Ohta--Ono's \emph{canonical de Rham model} for Floer theory, and so requires the ground field to contain $\R$.  It is likely that this requirement can be dropped using a suitable singular model but this has not been written down in detail.
\end{rmk}

Geometrically, the superpotential $W$ counts rigid pseudoholomorphic discs in $X$ with boundary on a fixed (but arbitrary) Lagrangian torus orbit (or \emph{toric fibre}) $L$, which send a boundary marked point to a generically chosen point $p \in L$.  In this count, each disc is weighted by $T^\text{area}$ and by its boundary homology class.  The summand $W_j$ counts the same discs, but additionally weighted by their intersection numbers with the toric divisor $D_j$.  These disc counts can be viewed as formal functions on the space $H^1(L; \C^*)$ of local systems on $L$, which we can pull back to $H^1(L; \C)$ via the exponential map, and the differential of $W$ as a function on $H^1(L; \C)$ satisfies
\[
\diff W= \sum_{j=1}^N \nu_j W_j
\]
(see \cref{logDerivative}).  This motivates:

\begin{defn}
Algebras of the form appearing in \eqref{eqGenJac} are \emph{generalised Jacobian rings}.
\end{defn}

By combining \cref{corModIso} with \cref{propFOOOks} we obtain the main geometric result of the paper, which is a form of \emph{closed-string mirror symmetry}:

\begin{mcor}
\label{corksIso}
For any toric variety $X$ whose moment polyhedron has a vertex, the Kodaira--Spencer map constructed geometrically by Fukaya--Oh--Ohta--Ono in \cite{FOOOMSToric} gives an isomorphism of $\Lambda_0$-algebras
\[
\ks : QH^*(X; \Lambda_0) \xrightarrow{\ \cong\ } \LHD / \clos ( \diff W ).
\]
\end{mcor}

The picture to have in mind is that $\Spec QH^*(X; \Lambda_0)$ is a branched cover of $\Spec \Lambda_0$, which is a kind of formal disc, whose fibre over $T=0$ is simply $\Spec$ of the classical cohomology $H^*(X; \C)$.  The point of \cref{Theorem1}---and of the definition of $\LHD$---is that $\Spec \LHD / \clos ( \diff W )$ is also a branched cover of $\Spec \Lambda_0$, with the same fibre over $0$.  The map $\Spec \ks$ defines a morphism between the two covers, and the classical Stanley--Reisner presentation ensures that it is an isomorphism over $0$, from which we deduce that it is an isomorphism everywhere.

Since our proof does not rely on the details of the construction of $\ks$, it can readily accommodate modifications such as bulk deformations and $B$-fields, which we discuss in \cref{sscBulk,sscBFields} (these basically amount to considering the \emph{big} quantum cohomology).  Exactly how the result should be interpreted for non-compact $X$ is discussed in \cref{sscNonCompact}---essentially, quantum cohomology now depends on the choice of almost complex structure, and the corollary holds for the standard integrable one. Although the Kodaira--Spencer map is defined without taking the closure of the ideal in the codomain, the closure is crucial in our proof that $\ks$ is an isomorphism, where it appears naturally upon taking a limit.

Various properties of quantum cohomology follow from the isomorphism of \cref{corksIso}, for example the invertibility of the toric divisor classes and \emph{generic semisimplicity} when $X$ is compact.  The former is well-known and we discuss it briefly in \cref{sscInvertibility}, whilst for the latter we refer the interested reader to \cite[Corollary 5.12]{IritaniConvergence}.

If $X$ is monotone (the first Chern class is a positive multiple of the K\"ahler class) then things simplify considerably; see \cref{MonotoneQH} and the surrounding discussion in \cref{sscMonotonicity}.  In particular, the superpotential $W$ coincides with its \emph{leading-order} term $\lv_1 + \dots + \lv_N$, so the presentation of $QH^*(X)$ is completely explicit.  We can also avoid the slightly awkward completions and closures which appear in the general case.  For non-compact monotone $X$, discussed in \cref{sscMonNonCpct}, one can use contact-type almost complex structures as well as the standard integrable one if $X$ is conical at infinity.

\begin{rmk}
\label{rmkLocJac}
Let $\Lambda$ denote the universal Novikov \emph{field}, defined in the same way as $\Lambda_0$ but with the exponents of $T$ now allowed to be negative.  When $X$ is compact, we will see in \cref{sscInvertibility} that
\[
\Lambda \llangle H \rrangle_\Delta \coloneqq \Lambda \otimes_{\Lambda_0} \LHD
\]
is a completion of the group ring $\Lambda [H_1(L; \Z)]$, and the reader may be more familiar with statements of closed-string mirror symmetry involving rings of this form.  When $X$ is non-compact, however, $\Lambda \llangle H \rrangle_\Delta$ is a completion of a \emph{proper subring} of the group ring.  Adjoining the missing monomials by localising in a suitable way is conjectured to give the \emph{symplectic cohomology}, $SH^*(X; \Lambda)$, and this has been verified in some special cases \cite{RitterFanoToric}.
\end{rmk}

To demonstrate \cref{corksIso}, or rather its monotone refinement \cref{MonotoneQH}, we give a quick sample calculation, which also illustrates \cref{rmkLocJac}:

\begin{ex}
\label{Example}
Let $X$ be the total space of the line bundle $\mathcal{O}_{\C\P^1}(-1)$, whose polyhedron has facets with normals
\[
\nu_1 = \begin{pmatrix} 1 \\ 0 \end{pmatrix} \text{, } \nu_2 = \begin{pmatrix} 1 \\ 1 \end{pmatrix} \text{ and } \nu_3 = \begin{pmatrix} 0 \\ 1 \end{pmatrix},
\]
and take $\lambda_j = 1$ for all $j$.  This is monotone, so its quantum cohomology can be defined over $\Z[T] \subset \Lambda_0$, and by \cref{MonotoneQH} there is an isomorphism of $\Z[T]$-algebras
\[
\ks_\mathrm{mon} : QH^*(X; \Z[T]) \rightarrow \Z[T, \lv_1, \lv_2, \lv_3] / (\sum_j \nu_j \lv_j),
\]
which sends each $H_j$ to $\lv_j$.  Here $\Z[T, \lv_1, \lv_2, \lv_3]$ is the monotone analogue of $\LHD$, and the $\lv_j$ are monomials in $\Z[T, x^{\pm 1}, y^{\pm 1}] \cong \Z[T][H_1(L; \Z)]$ whose $x$- and $y$-exponents are given by the components of the $\nu_j$.  Explicitly we have $\lv_1 = Tx$, $\lv_2 = Txy$ and $\lv_3 = Ty$, so upon inverting $T$ the ring $\Z[T, \lv_1, \lv_2, \lv_3]$ becomes a proper subring of the group ring $\Z[T^{\pm 1}][H_1(L; \Z)]$.

The relations between $T$ and the $\lv_j$ are generated by $\lv_1 \lv_3 = T \lv_2$, whilst the ideal $(\sum_j \nu_j \lv_j)$ imposes the extra relations $\lv_1 + \lv_2 = \lv_2 + \lv_3 = 0$.  The codomain of $\ks_\mathrm{mon}$ can therefore be written as $\Z[T, \lv_2] / (\lv_2^2 - T \lv_2)$, where $\lv_2$ is now considered as an independent formal variable.  Noting that $H_2$ is exactly the exceptional divisor class $E$, we obtain
\[
QH^*(X; \Z[T]) = \Z[T, E] / (E^2-TE).
\]
This agrees with the computation of Ritter \cite[Corollary 3]{RitterNegative}, after observing that his class $\omega_Q$ (which evaluates to $+1$ on the exceptional curve) is equal to $-E$.  Ritter also shows that
\[
SH^*(X; \Z[T^{\pm 1}]) = \Z[T^{\pm 1}, E] / (E-T),
\]
and after rewriting this as $\Z[T^{\pm 1}, E^{\pm 1}]/(E^2-TE)$ we see that in this case the symplectic cohomology is indeed obtained from the quantum cohomology by adjoining the missing group ring monomials.
\end{ex}

\cref{corksIso} is proved for compact $X$ in \cite{FOOOMSToric}, but (in contrast to the proof of \cref{propFOOOks}) the compactness of $X$ is used in an essential way.  The main contribution of this work is therefore a proof in the non-compact case, where there are few existing computations of quantum cohomology (see \cite{RitterFanoToric} for the state of the art).  Here even the statement of the correct form of the Jacobian ring on the right-hand side seems not to have appeared in the literature in full generality, although it may be known to experts.

Our method of proof also has significant advantages over previous approaches.  First, it involves no direct computation of relations in quantum cohomology (for example, the \emph{quantum Stanley--Reisner relations}, which are the key ingredient in \cite{FOOOMSToric}; they correspond to $E^2-TE$ in \cref{Example}), and no dimension-counting to prove injectivity of the map.  Second, our description of the algebra $\LHD$ is based on a certain monoid, which is completely natural from the perspective of Floer theory (it is in some sense the smallest monoid which can record the boundaries and areas of holomorphic discs bounded by $L$), whose geometry vividly encodes both the quantum Stanley--Reisner relations and the filtration used to ensure convergence of disc counts.  And finally our proof is rather self-contained, apart from an application of Reisner's Cohen--Macaulay criterion: we do not even assume the presentation \eqref{SRPresentation} of the classical cohomology of $X$, showing instead how it follows from the geometrically well-motivated \cref{propFOOOks} (in fact, from just its leading-order part, which is much easier to construct---see \cref{rmkSRObjections}).

A secondary goal of this paper is to make the ideas behind \cref{corksIso} accessible to a wider audience, and particularly to algebraic geometers, by focusing on the algebra and keeping the discussion of the pseudoholomorphic curve theory at a schematic level.  Our definitions of the main rings and their topologies differ slightly from those given by Fukaya--Oh--Ohta--Ono, and hopefully emphasise the geometric intuition behind them.

We expect that \cref{Theorem1}, and the ideas behind it, will have other applications beyond the proof of \cref{corksIso}.  For instance, one may be able to give alternative geometric constructions of algebra maps satisfying the hypotheses of \cref{corModIso}, perhaps incorporating extensions such as equivariant parameters.  More generally, the subring of the cohomology of a variety generated by the components of a normal crossings divisor is naturally a quotient of a Stanley--Reisner ring (the $Z_j$ correspond to components and the vanishing of products corresponds to emptiness of intersections) and we hope that our approach to \cref{Theorem1} may be brought to bear in understanding deformations of such rings, arising in Floer theory or otherwise.

\subsection{Structure of the paper}

In the rest of the introduction we sketch the motivation behind \cref{corksIso} for readers unfamiliar with mirror symmetry (\cref{sscMSMotivation}) and discuss previous related results (\cref{OtherWorks}).  \Cref{secks} then outlines the construction of the Kodaira--Spencer map, beginning with a quick recap on toric geometry.  Readers already comfortable with the idea, or who are only interested in the algebraic statement \cref{Theorem1}, may wish to skip this, consulting only \cref{sscCone,sscCoeffRings} for the definition of the ring $\LHD$ and elements $\lv_1, \dots, \lv_N$ appearing in the main results.

The proof of \Cref{Theorem1} occupies \cref{secProof}.  In \cref{sscSR} we relate $\LHD$ to the Stanley--Reisner ring and give a novel proof of the Stanley--Reisner presentation, which may be of independent interest.  In \cref{sscMainLemmas} we then prove the two main lemmas, namely that the Stanley--Reisner ring is Cohen--Macaulay (in fact, Gorenstein if $X$ is compact) and that the generators of the ideal $(\sum_j \nu_j Z_j)$ form a regular sequence.  \Cref{sscProof} combines these with an algebraic `finitisation' procedure to prove \cref{Theorem1}.

We finish in \cref{secFurtherDiscussion} by covering some technical points and variants of our results.  \Cref{sscNonCompact} explains how to interpret \cref{corksIso} when $X$ is non-compact---in this case one has to be careful about pseudoholomorphic curves escaping to infinity.  Next we discuss the refinements that can be made when $X$ is monotone (\cref{sscMonotonicity}) and how to deal with non-compactness in this setting (\cref{sscMonNonCpct}).  We then deal with modifications arising from bulk deformations (\cref{sscBulk}) and $B$-fields (\cref{sscBFields}), and finally prove the invertibility of the toric divisors when $X$ is compact (\cref{sscInvertibility}).

\subsection{Mirror symmetry motivation}
\label{sscMSMotivation}

Mirror symmetry is fundamentally a duality between torus fibrations.  In its most basic form, we start with a lattice $M \cong \Z^n$ inside the vector space
\[
M_\R = M \otimes_\Z \R
\]
and consider the dual torus fibrations $Y$ and $\check{Y}$ over $M_\R$ with fibres $M_\R^* / M^*$ and $M_\R / M$.  These are fibrewise quotients of the cotangent and tangent bundles of $M_\R$, hence carry natural symplectic and complex structures respectively, and constitute the archetypal example of a mirror pair.

The symplectic manifold $Y$ is fibred by (special) Lagrangian tori, whilst in accordance with the \emph{SYZ} philosophy (see \cite{Au} for a beautiful introduction) its mirror $\check{Y}$ parametrises pairs $(L, \rho)$ comprising one of these tori $L$ and a choice of $\mathrm{U}(1)$-local system
\[
\rho : H_1(L; \Z) \rightarrow \mathrm{U}(1).
\]
Each $(L, \rho)$ defines an object of the Fukaya category $\mathcal{F}(Y)$ of $Y$, and since the fibres bound no pseudoholomorphic discs we have
\[
H^*\Hom_{\mathcal{F}(Y)}((L_0, \rho_0), (L_1, \rho_1)) = HF^*((L_0, \rho_0), (L_1, \rho_1)) \cong \begin{cases} H^*(L_0) & \text{if } (L_0, \rho_0) = (L_1, \rho_1) \\ 0 & \text{otherwise.} \end{cases}
\]
The latter isomorphism follows from the fact that if $L_0 \neq L_1$ then the Floer complex
\[
CF^*((L_0, \rho_0), (L_1, \rho_1))
\]
vanishes since the Lagrangians are disjoint, whilst if $L_0 = L_1$ then it reduces to the Morse complex of $L_0$ twisted by the local system $\rho_0^* \otimes \rho_1$; here the $*$ denotes the dual representation, or equivalently the complex conjugate.  From the perspective of \emph{homological} mirror symmetry we would say that $\mathcal{F}(Y)$ is equivalent in an appropriate sense to the derived category $\mathrm{D}^b \mathrm{Coh} (\check{Y})$ of coherent sheaves on $\check{Y}$, with a torus object in the former sent to the skyscraper sheaf of the corresponding point in the latter.

\begin{rmk}
\label{WrappedCat}
Since $Y$ is non-compact, in order to get an equivalence with $\mathrm{D}^b \mathrm{Coh} (\check{Y})$ we should really work with the \emph{wrapped} Fukaya category $\mathcal{W}(Y)$, in which certain non-compact Lagrangians are allowed, and the Floer complexes between such Lagrangians are defined by `wrapping' one of them around the conical end of $Y$ as it goes off to infinity.
\end{rmk}

Now suppose that we compactify $Y \cong (\C^*)^n$ to the symplectic toric variety $X$ considered in \cref{sscOverview} (but now assumed compact).  This introduces pseudoholomorphic discs bounded by the fibres, and heuristically we can define a holomorphic function
\[
W : \check{Y} \cong (\C^*)^n \rightarrow \C
\]
by sending a pair $(L, \rho)\in \check{Y}$ to the count of rigid pseudoholomorphic discs $u$ with boundary on $L$, weighted by the monodromy $\rho ([\pd u])$.  After fixing a choice of $L$ we can---still heuristically---view this function as an element of $\C[H_1(L; \Z)]$.  In reality, since there may be infinitely many discs we should work over the Novikov field $\Lambda$ rather than $\C$ and additionally weight each disc by $T^\text{area}$ to ensure that the count converges.  This makes $W$ into an element of a suitable completion $\Lambda \llangle H \rrangle_\Delta$ of $\Lambda[H_1(L; \Z)]$, and this element is exactly the superpotential appearing in \cref{propFOOOks}.

The presence of these discs reduces the $\Z$-grading on our categories to $\Z/2$, and makes most of the torus fibres Floer-theoretically trivial: only those corresponding to critical points of $W$ survive (see \cref{logDerivative}).  The prediction of mirror symmetry is now that (once it is rigorously defined) the Fukaya category $\mathcal{F}(X)$ should be equivalent to the derived category $\mathrm{DMF}(\check{Y}, W)$ of matrix factorisations of $(\check{Y}, W)$, and, in particular, that the Hochschild cohomologies of these two categories should coincide.

The Hochschild cohomology of the Fukaya category is expected to be isomorphic to the quantum cohomology of $X$ via the \emph{closed--open string map}
\[
\CO : QH^*(X) \rightarrow HH^*(\mathcal{F}(X))
\]
which we discuss in more detail in \cref{sscksSketch}.  That of the matrix factorisation category, meanwhile, should be the Jacobian ring of the superpotential $W$.  This has been proved in various situations (see the work of Dyckerhoff \cite{Dyckerhoff}, Segal \cite{Segal} and C\u ald\u araru--Tu \cite{CaldararuTu} for instance), but the present setting is slightly non-standard due to the completed rings involved.  This leads to the prediction that the quantum cohomology of $X$ should be isomorphic to the Jacobian ring of $W$ via a closed--open string map, which is essentially what \cref{corksIso} verifies.

If $X$ is non-compact then, as in \cref{WrappedCat}, we should instead work with the wrapped Fukaya category $\mathcal{W}(X)$.  We should also replace quantum cohomology with symplectic cohomology (the `wrapped' equivalent) in the closed--open map, which yields the informal conjectural isomorphism
\begin{equation}
\label{SHconj}
SH^*(X; \Lambda) \rightarrow (\text{completion of } \Lambda[H_1(L; \Z)]) / \clos (\diff W)
\end{equation}
stated in \cref{rmkLocJac}.  There is expected to be an \emph{acceleration} map $QH^*(X) \rightarrow SH^*(X)$ (and a corresponding functor $\mathcal{F}(X) \rightarrow \mathcal{W}(X)$; see \cite{RitterSmith}) and \cref{corksIso} identifies its domain as the right-hand side of \eqref{SHconj} but with the `numerator' replaced by its subring $\Lambda \llangle H \rrangle_\Delta$.

\subsection{Relation to other works}
\label{OtherWorks}

Batyrev \cite{BatyrevToricQH} and Givental \cite{GiventalICM} first suggested that when $X$ is compact and Fano $QH^*(X; \C)$ should be isomorphic to the Jacobian ring of a superpotential $W$ in $\C[y_1^{\pm 1}, \dots, y_n^{\pm 1}]$, which they explicitly wrote down.  This prediction became known as closed-string mirror symmetry, and in our language their $W$ coincides with the leading-order term $\lv_1 + \dots \lv_N$.  The Fano hypothesis is the algebro-geometric analogue of monotonicity.

Batyrev proved that the Jacobian ring is isomorphic to a combinatorially-defined ring---which we shall call the \emph{Batyrev quantum cohomology} of $X$, denoted $\QHB^*(X)$---given as the quotient of a polynomial ring by quantum Stanley--Reisner relations and certain linear relations.  He also sketched an argument that $QH^*_B(X)$ it is isomorphic to the usual quantum cohomology of $X$, and this was later proved by Givental \cite{GiventalToricCI} using torus-localisation on moduli spaces of holomorphic curves (made rigorous in \cite{GraberPandharipande}).  Combining these results therefore gives a proof of closed-string mirror symmetry for compact toric Fanos, but by the indirect method of computing each side separately and equating both with $\QHB^*(X)$.

\begin{rmk}
In our approach, relations between the $\lv_j$ in $\LHD$ arise from the geometry of a monoid $\Gamma_\R$ used to define it in \cref{sscCoeffRings}, and naturally translate into the quantum Stanley--Reisner relations.  Batyrev's linear relations are replaced by the generators of the ideal $(\sum_j \nu_j \hv_j)$, which are linear in the $\lv_j$ to leading order.
\end{rmk}

Subsequent to the original work of Givental, alternative proofs of the isomorphism between $QH^*(X)$ and $\QHB^*(X)$ in the Fano case have been given by Cieliebak--Salamon \cite{CieliebakSalamon} and McDuff--Tolman \cite{McDuffTolman}.  The former includes the hypothesis that the minimal Chern number of $X$ is at least $2$, but does not assume that $X$ is compact.  In both cases the heart of the argument is to show that the quantum Stanley--Reisner relations hold in $QH^*(X)$ so that one can define a homomorphism $\QHB^*(X) \rightarrow QH^*(X)$.  That this is an isomorphism then follows from a formal algebraic argument \cite[Lemma 5.1]{McDuffTolman} using the fact that the relations defining $\QHB^*(X)$ are perturbations of those defining the classical Stanley--Reisner presentation of $H^*(X)$.  Ritter has generalised the approach of McDuff--Tolman, based on the Seidel representation, to \emph{admissible} non-compact toric Fanos \cite{RitterFanoToric}, simplifying and extending previous computations of his using virtual localisation \cite{RitterNegative}.

Recently, Fukaya--Oh--Ohta--Ono \cite{FOOOMSToric} gave a more conceptual proof of closed-string mirror symmetry for arbitrary compact toric varieties along the lines of \cref{corksIso}, which we have already mentioned.  In their version, the generalised Jacobian ring appearing in \cref{corksIso} is replaced by a similar ring which they denote by $\Jac \mathfrak{PO}$ (they call the superpotential $\mathfrak{PO}$ rather than $W$).  After adapting McDuff--Tolman's argument to prove the quantum Stanley--Reisner relations and show that (roughly) $QH^*(X)$ is a quotient of $\QHB^*(X)$, they then show that $\dim \QHB^*(X) = \dim \Jac \mathfrak{PO}$ and deduce that $\ks$ is injective (surjectivity is relatively easy).  The equality of dimensions is established by giving a finite flat family of schemes which has both $\Spec QH^*(X)$ and $\Spec \Jac \mathfrak{PO}$ as fibres.  In the compact Fano case this equality can be deduced from Kushnirenko's theorem, as observed by Ostrover--Tyomkin \cite{OstroverTyomkin}, but this approach is rather ungeometric, and provides no map between $\Jac \mathfrak{PO}$ and $H^*(X)$.

We end by remarking that quantum cohomology of toric varieties has also been studied intensively, and largely independently, from an algebro-geometric perspective.  Many similar results have been proved in this setting, but at present their relationship with the analogous symplectic results is unclear.  For instance, Iritani \cite{IritaniBigQH} computed the big equivariant quantum cohomology of (possibly non-compact) toric varieties using the algebro-geometric version of the Seidel representation (see \cite{IritaniShiftOperators}), and his mirror map should be related to the superpotential of Fukaya--Oh--Ohta--Ono which appears in our results.  In algebraic geometry the Seidel representation exists for formal reasons, using torus localisation on moduli spaces, and does not require the delicate analysis carried out by Ritter in the symplectic case.  In a different direction, Gonz\'alez--Woodward \cite{GonzalezWoodwardToric} computed the quantum cohomology of compact toric \emph{orbifolds} using the quantum Kirwan map of \cite{WoodwardQuantumKirwanI}.  This is motivated by the study of symplectic vortices but again relies on algebro-geometric properties of moduli spaces.

\subsection{Acknowledgements}

I am grateful to Jonny Evans, Yank\i~Lekili, Ivan Smith and Richard Thomas for helpful comments and feedback, and to  Kenji Fukaya, Hiroshi Iritani, Kaoru Ono, Alex Ritter and Chris Woodward for useful conversations and correspondence about their work.  The exposition has benefited greatly from the criticisms of an anonymous referee.  This work was funded by EPSRC Grant EP/P02095X/1.

\section{The Kodaira--Spencer map}
\label{secks}

\subsection{Toric geometry background}
\label{sscToricBG}

For us, a toric variety $X$ is defined by a \emph{Delzant polyhedron} $\Delta \subset \mathfrak{t}^*$, where $\mathfrak{t}$ is the Lie algebra of an abstract $n$-torus.  Explicitly, if $A \cong \Z^n$ denotes the lattice in $\mathfrak{t}$ given by the kernel of the exponential map, then $\Delta$ is a subset of $\mathfrak{t}^*$ of the form
\begin{equation}
\label{polyhedronDefn}
\Delta = \{ x \in \mathfrak{t}^* : \ip{x}{\nu_j} \geq -\lambda_j \text{ for } j = 1, \dots, N\},
\end{equation}
where $\nu_1, \dots, \nu_N$ are elements of $A$ and $\lambda_1, \dots, \lambda_N$ are positive real numbers, such that the following conditions are satisfied (see \cite[Definition 1.2]{ProudfootToric} or the original paper of Delzant \cite{Delzant}): exactly $d$ facets meet at each codimension-$d$ face of $\Delta$, and the normals $\nu_j$ to these $d$ facets extend to a basis for the free abelian group $A$.  We will assume that none of the inequalities in \eqref{polyhedronDefn} are redundant, so that $\Delta$ has exactly $N$ facets.  We also assume that $\Delta$ has at least one vertex but reiterate that it need not be compact.

\begin{rmk}
\label{VertexRmk}
Suppose we didn't impose that condition that $\Delta$ has a vertex.  Let $A^*_\text{an}$ be the annihilator of $\nu_1, \dots, \nu_N$ inside the dual lattice $A^*$, and let $k$ denote its rank.  If $A^*_\text{comp}$ denotes a complement to $A^*_\text{an}$ then $A^*$ splits as $A^*_\text{an} \oplus A^*_\text{comp}$, and $\Delta$ correspondingly splits as $\Delta_\text{an} \oplus \Delta_\text{comp}$, where $\Delta_\text{an}$ is $A^*_\text{an} \otimes \R \cong \R^k$ and $\Delta_\text{comp} \subset A^*_\text{comp} \otimes \R$ is the Delzant polyhedron with normals given by $\nu_j \mod (A^*_\text{comp})^{\perp}$.  These normals have trivial annihilator in $A^*_\text{comp}$, so $\Delta_\text{comp}$ contains no affine line and hence must have at least one vertex.  Then $X$ decomposes as $(\C^*)^k \times X_\text{comp}$, where $X_\text{comp}$ is the toric variety defined by $\Delta_\text{comp}$.  In particular we see that $QH^*(X)$ is the tensor product of an exterior algebra on $k$ variables (of degree $1$) and $QH^*(X_\text{comp})$, which can be computed by our methods, so the assumption that $\Delta$ has a vertex is no real loss of generality from the perspective of computing quantum cohomology.
\end{rmk}

The construction of $X$ from $\Delta$ is described in the two previous references, and, briefly, is as follows.  Equip $\C^N$ (with coordinates $(w_1, \dots, w_N)$) with the symplectic form
\[
\frac{i}{2} \sum_{j=1}^N \diff w_j \wedge \diff \overline{w}_j,
\]
and consider the obvious action of the $N$-torus $T^N$ with moment map
\begin{equation}
\label{MomMap}
(w_1, \dots, w_N) \mapsto \lb\frac{1}{2}|w_1|^2 - \lambda_1, \dots, \frac{1}{2}|w_N|^2 - \lambda_N\rb \in \R^N \cong \mathrm{Lie}(T^N)^*.
\end{equation}
Let $K$ be the kernel of the homomorphism $\Z^N \rightarrow A$ which sends the $j$th basis vector to $\nu_j$, and let $T_K$ be the subtorus of $T^N = \R^N / \Z^N$ given by $(\R \otimes K) / K$.  Define $X$ to be the symplectic reduction of $\C^N$ with respect to this $T_K$-action, at the zero level of the moment map restricted from \eqref{MomMap}.

The space $X$ is a symplectic manifold with symplectic form $\omega'$ induced from $\C^N$.  It also inherits an action of the quotient $T^N / T_K$, which is precisely the abstract $n$-torus from above, with Lie algebra $\mathfrak{t}$.  The image of the moment map for this action is $\Delta$, and in \cref{sscOverview} we defined the toric divisors $D_1, \dots, D_N$ to be the moment map preimages of the facets $F_1, \dots, F_N$.  These divisors define classes in Borel--Moore homology $\HBM_*(X; \Z)$, and their Poincar\'e duals are $H_1, \dots, H_N$ in $H^2(X; \Z)$.  The neighbourhood of a point in $\Delta$ which lies on exactly $k$ facets looks like a neighbourhood of $0$ in $\R_{\geq 0}^k \times \R^{n-k}$, and its moment map preimage looks like a neighbourhood of $\{(0, \dots, 0)\} \times \mathrm{U}(1)^{n-k}$ in $\C^k \times (\C^*)^{n-k}$.  In this preimage, the toric divisors are the hyperplanes defined by setting each of the first $k$ coordinates equal to zero in turn.  In particular, the intersection of any collection of toric divisors is transverse.

The algebro-geometric description of $X$ is as follows (see \cite[Section 1.4]{ProudfootToric}).  Consider the cone over $\Delta \times \{1\}$ inside $\mathfrak{t}^* \times \R_{\geq 0}$ and let $\Sigma$ denote its closure.  Let $\Sigma_\Z$ be the monoid given by the intersection of $\Sigma$ with the lattice $A^* \times \Z_{\geq 0}$, and consider the monoid ring $\C[\Sigma_\Z]$, graded by the $\Z_{\geq 0}$ component of $\Sigma_\Z$.  As a variety, $X$ is given by $\Proj \C[\Sigma_\Z]$.  Its first Chern class is $H_1 + \dots + H_N$, and the symplectic form $\omega'$ is K\"ahler for the natural complex structure.

The superpotential $W$ and the Kodaira--Spencer map $\ks$ are both defined by counting pseudoholomorphic discs in $X$ bounded by a torus orbit $L$, and for simplicity we take $L$ to be the moment map preimage of $0$.  This is why we assumed that the $\lambda_j$ are all strictly positive; by translating $\Delta$ (which does not affect $X$) this can always be achieved.  Note that there is a canonical identification between $H_1(L; \Z)$ and the lattice $A \subset \mathfrak{t}$, so an element of $A$ can be viewed both as a loop on $L$ and as the generator of a one-parameter subgroup of the torus which acts on $L$, and the loop is the orbit of this subgroup with period $2\pi$.

The virtual dimension of the moduli space of parametrised pseudoholomorphic discs in a class $\beta \in H_2(X, L; \Z)$ is given by $n + \mu(\beta)$, where $\mu(\beta)$ is the \emph{(Maslov) index} of $\beta$---this is an even integer which is a relative version of twice the Chern number of a closed curve, and in the present toric context is given by twice the intersection number of $\beta$ with the sum of the toric divisors.  In particular, unparametrised discs with a boundary marked point constrained to a fixed $p \in L$ are rigid when they have index $2$ (the boundary marked point adds $1$ to the dimension of the moduli space, and quotienting by reparametrisation subtracts $3$).

With respect to the standard complex structure, there are $N$ families of smooth holomorphic discs of index $2$, and they are all Fredholm regular.  Each family comprises a single $T^N / T_K$-orbit, and a representative of the $j$th family is given by
\[
u_j(z) = e^{-i\nu_j \log z} \cdot p,
\]
for $z$ in the closed unit disc in $\C$ (when $z=0$ the right-hand side should be interpreted as the obvious limit).  Here $e^\cdot$ is the exponential map from the complexification of $\mathfrak{t}$ to the complexified torus $T_\C^N / (T_K)_\C$, which acts holomorphically on $X$.  The point $u_j(0)$ lies in the divisor $D_k$ if and only if $j=k$.  Note that the boundary of $u_j$ is the orbit of the group generated by $\nu_j$, so its homology class is exactly $\nu_j \in A = H_1(L; \Z)$.  The area of $u_j$ can be computed to be
\[
-2\pi\ip{\mu(u_j(0))}{\nu_j} = 2\pi \lambda_j.
\]
The reader is referred to the work of Cho--Oh \cite{ChoOh} for proofs and more details.

\begin{rmk}
\label{SympForm}
To avoid factors of $2\pi$ littering our discussion, we work with the symplectic form $\omega$ on $X$, defined to be $\omega' / 2\pi$.
\end{rmk}

Topologically, the important property of $X$ is the Stanley--Reisner presentation of its cohomology ring.  We will see in \cref{propSRPresentation} that this can be deduced from \cref{propFOOOks}, but we will independently need the well-known fact that the cohomology ring is generated by the classes $H_j$.  This is usually proved as a corollary of the Stanley--Reisner presentation, so for completeness we give a separate short proof:

\begin{lem}
\label{DjGenerate}
The classical cohomology ring $H^*(X; \Z)$ is generated by the classes $H_j$ Poincar\'e dual to the toric divisors.
\end{lem}
\begin{proof}
We will prove the equivalent statement that the Borel--Moore homology is spanned by the classes of invariant subvarieties (preimages of faces of the moment polyhedron).  To deduce the result we actually want, simply note that each invariant subvariety is a transverse intersection of toric divisors $D_j$, so its Poincar\'e dual is the cup product of the corresponding $H_j$.  The key step is to produce recursively a chain of subcomplexes (meaning unions of faces)
\begin{equation}
\label{SubCxes}
\Delta= \Delta_0 \supset \Delta_1 \supset \dots \supset \Delta_m = \emptyset
\end{equation}
of $\Delta$ such that each $\Delta_{j+1}$ is obtained by taking a vertex $x_j$ in $\Delta_j$, near to which $\Delta_j$ looks like a neighbourhood of a vertex in a simplex (of dimension $n_j$, say), and deleting the interior of the star of $x_j$.

Given such a chain of subcomplexes, let $X_j$ denote the moment map preimage of $\Delta_j$ in $X$.  Then $X_j \setminus X_{j+1}$ is homeomorphic to $\C^{n_j}$ and the inclusion/restriction exact sequence in Borel--Moore homology (over $\Z$) \cite[IX.2.1]{IversenSheaves} yields
\[
\dots \rightarrow \HBM_{i+1}(\C^{n_j}) \rightarrow \HBM_i(X_{j+1}) \rightarrow \HBM_i(X_j) \rightarrow \HBM_i(\C^{n_j}) \rightarrow \cdots.
\]
Supposing that $\HBM_*(X_{j+1})$ is concentrated in even degree, we deduce that the inclusion $X_{j+1} \hookrightarrow X_j$ induces an injection in degree $2n_j$ with cokernel $\Z$, and induces isomorphisms in all other degrees.  We are then down by decreasing induction on $n$ if we can show that the pullback on $\HBM_{2n_j}$ induced by the inclusion of $X_j \setminus X_{j+1} \cong \C^{n_j}$ into the toric $2n_j$-manifold given by the closure $Y_j$ of $X_j \setminus X_{j+1}$ in $X$ is an isomorphism (this shows that the cokernel in degree $2n_j$ comes from the fundamental class of the invariant subvariety $Y_j$).  This result follows from another application of the inclusion/restriction exact sequence, using the fact that $Y_j \cap X_{j+1}$ has Borel--Moore homology concentrated in degrees at most $2(n_j-1)$.

It is therefore left to build the chain of subcomplexes \eqref{SubCxes}.  In order to do this, fix a vertex of $\Delta$ and reorder the $\nu_j$ so that the normals to the facets at this vertex are $\nu_1, \dots, \nu_n$.  Choose positive real numbers $\lambda_1, \dots, \lambda_n$ which are linearly independent over $\Q$ and let $\nu$ denote $\lambda_1 \nu_1 + \dots + \lambda_n \nu_n$.  At the $j$th step, let $x_j$ be the unique point in $\Delta_j$ which minimises $\ip{x}{\nu}$.  This is necessarily a vertex of $\Delta_j$, and $x_0$ is precisely the vertex with normals $\nu_1, \dots, \nu_n$.  The only thing we need to check is that $\Delta_j$ looks like a simplex near $x_j$, and to see this observe that $\Delta$ looks like an $n$-simplex near $x_j$, and to get to $\Delta_j$ we simply restrict to faces which are contained in the closed half space on which $\ip{\cdot}{\nu} \geq \ip{x_j}{\nu}$.
\end{proof}

This arguments actually gives a bit more information:

\begin{lem}
\label{rmkHRank}
The $\Z$-module $H^*(X; \Z)$ is free of rank equal to the number of vertices of $\Delta$.  Moreover, the rank of $H^2(X; \Z)$ is $N-n$.
\end{lem}
\begin{proof}
The first part is immediate from the proof of \cref{DjGenerate}.  To prove the second part, note that in passing from $\Delta_0$ to $\Delta_1$ we remove $n$ facets of $\Delta$ (along with the interior of $\Delta$), and this step gives rise to the fundamental class in $\HBM_{2n}(X; \Z) \cong H^0(X; \Z)$.  The remaining $N-n$ facets are removed one at a time, and on each occasion we gain a generator in $\HBM_{2n-2}(X; \Z) \cong H^2(X; \Z)$; moreover, all generators of $H^2(X; \Z)$ arise in this way.
\end{proof}

Since the $H_j$ are represented by cycles $D_j$ in the complement of $L$, we can actually think of them as classes in $H^2(X, L; \Z)$, and we will often do this without explicit warning.  We then have the following:

\begin{cor}
\label{DiscDivisorPairing}
The classes $[u_j]$ and $H_j$ form free free bases for $H_2(X, L; \Z)$ and $H^2(X, L; \Z)$ respectively, and are dual under the natural pairing
\[
H^2(X, L; \Z) \times H_2(X, L; \Z) \rightarrow \Z.
\]
\end{cor}
\begin{proof}
First note that from the first part of \cref{rmkHRank} and the universal coefficient theorem, all homology and cohomology groups of $X$ are free $\Z$-modules, and that cohomology is dual to homology.  The same is also true of the torus $L$, and hence (by the long exact sequence) of the pair $(X, L)$ as well.

For each $k$ we have seen that the disc $u_k$ meets $D_j$ if and only if $j=k$, and that its Maslov index (twice its intersection with $D_1+\dots+D_N$) is $2$, so we have
\[
\ip{H_j}{[u_k]} = \delta_{jk}.
\]
This tells us that $H_2(X, L; \Z)$ and $H^2(X, L; \Z)$ split as
\begin{gather*}
\label{H2Split}
H_2(X, L; \Z) = \lspan{[u_1]} \oplus \dots \oplus \lspan{[u_N]} \oplus \lspan{H_1, \dots, H_N}^\perp
\\ H^2(X, L; \Z) = \lspan{H_1} \oplus \dots \oplus \lspan{H_N} \oplus \lspan{[u_1], \dots, [u_N]}^\perp,
\end{gather*}
and by the previous paragraph the two $\perp$ summands are free $\Z$-modules of equal rank.  Considering the exact sequence
\[
\dots \rightarrow H^1(L; \Z) \rightarrow H^2(X, L; \Z) \rightarrow H^2(X; \Z) \rightarrow \cdots
\]
and applying the second part of \cref{rmkHRank} we conclude that these summands actually vanish.
\end{proof}

\subsection{Outline}

The purpose of the remainder of this section is to describe the geometry of the Kodaira--Spencer map $\ks$ appearing in \cref{propFOOOks}.  We follow Fukaya--Oh--Ohta--Ono \cite{FOOOMSToric}, who first defined it in its full generality (although they nominally restrict their attention to compact $X$) and introduced the name and the notation $\ks$.  For monotone $X$ the same idea also appears in the work of Biran--Cornea \cite[Section 7.3]{BCEG}, and in \cref{sscMonotonicity} we discuss the simplifciations and refinements that occur in this case.

The major steps are as follows:
\begin{enumerate}[(i)]
\item\label{itmVFCs} Construct virtual fundamental chains on moduli spaces of pseudoholomorphic curves so that counts of such curves can be defined.  The cycles must satisfy certain compatibility conditions in order to prove algebraic relations between the counts using cobordisms between the corresponding moduli spaces.
\item\label{itmRings} Define the topological ring $\LHD$ in which the disc counts are supposed to live and verify that the counts do indeed give rise to well-defined elements of this ring.  This involves showing that the counts are sums of monomials contained in the ring, and that these sums converge.
\item\label{itmSuperpot} Define the superpotential $W$ and its summands $W_1, \dots, W_N$.
\item\label{itmLeading} Verify that the leading-order term of each $W_j$, meaning its reduction modulo $\Lambda_+$, is $\lv_j$.
\item\label{itmks} Define the map $\ks$ and show that it is a ring homomorphism and that $\ks(H_j) = W_j$.
\end{enumerate}

We do not address \eqref{itmVFCs} at all and simply take it as a black box provided by \cite{FOOOMSToric}.  It involves difficult equivariant transversality arguments and represents a major technical achievement.  Instead, we begin in \cref{sscksSketch} by explaining \eqref{itmSuperpot} and \eqref{itmks}, after a broader discussion to provide some motivation and context for the construction of $\ks$.  Then in \cref{sscCone,sscCoeffRings} we cover \eqref{itmRings}, and finally in \cref{sscLeading} we deal with \eqref{itmLeading}.  Throughout, we focus on the geometric intuition, with \eqref{itmVFCs} providing the rigorous justification that the necessary curves can be counted and the cobordism arguments made precise.

\subsection{The construction}
\label{sscksSketch}

Modulo technical foundations, the (compact) Fukaya category $\mathcal{F}(X)$ of a symplectic manifold $X$ is an $A_\infty$-category over $\Lambda$.  Its objects are compact Lagrangian submanifolds $L$ of $X$, satisfying technical conditions (for example, weakly unobstructed, orientable and relatively spin) and decorated with various extra data (say, a weak bounding cochain, orientation and relative spin structure, and local system), with morphisms given by Floer cochain complexes.  The closed--open string map is a (still largely conjectural) $\Lambda$-algebra homomorphism $\CO$ from the quantum cohomology of $X$ to the Hochschild cohomology of $\mathcal{F}(X)$ \cite{SeidelICM}.  By projecting the Hochschild complex to its length-zero part we obtain for each object $\mathbf{L}$ (denoting a Lagrangian satisfying the necessary conditions and equipped with the extra choices) a $\Lambda$-algebra homomorphism
\[
\CO^0 : QH^*(X; \Lambda) \rightarrow HF^*(\mathbf{L}, \mathbf{L}; \Lambda),
\]
to the endomorphism algebra of $\mathbf{L}$ in the cohomological category $H \mathcal{F}(X)$.

Given a cohomology class $\alpha$ on $X$ which is Poincar\'e dual to a cycle $Z$, one considers the moduli space of pseudoholomorphic discs which map an interior marked point to $Z$.  The class $\CO^0(\alpha)$ is then, roughly speaking, defined to be Poincar\'e dual to the cycle on $L$ swept from this moduli space by a boundary marked point, weighted by the monodromy of the local system.  The fact that this is a ring homomorphism is proved by a cobordism argument, schematically illustrated in \cref{figDiscDeg} (which is essentially a simplified version of \cite[Fig.~2]{SherFan}).
\begin{figure}[ht]
\centering
\begin{tikzpicture}[blob/.style={circle, draw=black, fill=black, inner sep=0, minimum size=\blobsize}, x=1cm, y={(-0.05cm, 0.32cm)}, z={(0cm, 1cm)}]
\def\blobsize{1.5mm}
\def\rad{2}
\def\prad{1}
\def\srad{1.2} 
\def\drad{1.6} 
\def\sdrad{1.4} 
\def\theta{60}
\def\phi{80}

\draw (0, 0, 0) circle[radius=\rad];
\draw (\rad, 0, 0) node[blob](p){};
\foreach \n in {0, 1}
{
\draw ($\prad*cos(\phi+\n*180)*(1, 0, 0)+\prad*sin(\phi+\n*180)*(0, 1, 0)$) node[blob](a\n){};
}
\draw ($(p)+(0.7, 0)$) node{output};
\draw ($(a0)+(-0.35, -0.1)$) node{$Z_1$};
\draw ($(a1)+(-0.35, -0.1)$) node{$Z_2$};

\begin{scope}[shift={(-4.8, -6)}]
\draw (0, 0, 0) circle[radius=\rad];
\draw (\rad, 0, 0) node[blob](p){};

\foreach \n in {0}
{
\draw ($\srad*cos(\phi+\n*180)*sin(\theta)*(1, 0, 0)+\srad*sin(\phi+\n*180)*sin(\theta)*(0, 1, 0)+\srad*cos(\theta)*(0, 0, 1)+(0, 0, \srad)$) node[blob](a\n){};
}

\draw ($(a0)+(-0.35, -0.1)$) node{$Z_1$};

\sphere{(0, 0, \srad)}{\srad}

\foreach \n in {1}
{
\draw ($\srad*cos(\phi+\n*180)*sin(\theta)*(1, 0, 0)+\srad*sin(\phi+\n*180)*sin(\theta)*(0, 1, 0)+\srad*cos(\theta)*(0, 0, 1)+(0, 0, \srad)$) node[blob](a\n){};
}

\draw ($(p)+(0.7, 0)$) node{output};
\draw ($(a1)+(-0.35, -0.1)$) node{$Z_2$};
\end{scope}

\begin{scope}[shift={(4.2, -4)}]

\draw ($1.8*\sdrad*cos(\phi)*(1, 0, 0)+1.8*\sdrad*sin(\phi)*(0, 1, 0)$) node[blob](a0){} circle[radius=0.8*\sdrad];
\draw (0, 0) circle[radius=\sdrad];
\draw ($1.8*\sdrad*cos(\phi+180)*(1, 0, 0)+1.8*\sdrad*sin(\phi+180)*(0, 1, 0)$) node[blob](a1){} circle[radius=0.8*\sdrad];

\draw (\sdrad, 0, 0) node[blob](p){};
\draw ($(p)+(0.7, 0)$) node{output};
\draw ($(a0)+(-0.35, -0.1)$) node{$Z_1$};
\draw ($(a1)+(-0.35, -0.1)$) node{$Z_2$};
\end{scope}

\draw[dashed, ->] (-2.15, -1) -- (-3.15, -2);
\draw[dashed, ->] (2.15, -1) -- (3.15, -2);

\end{tikzpicture}
\caption{$\CO^0$ is a ring homomorphism.\label{figDiscDeg}}
\end{figure}
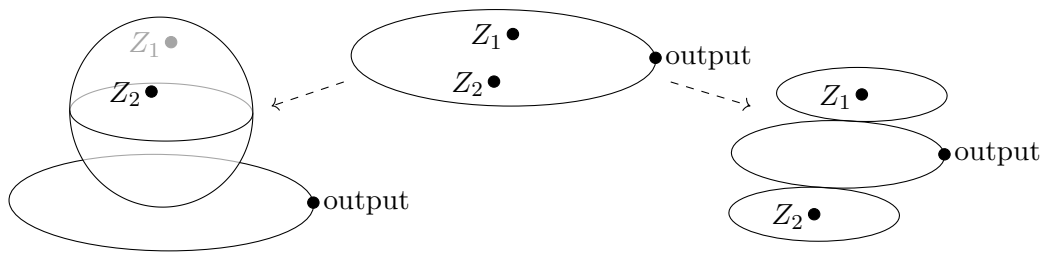
We start with a disc with two interior marked points which are constrained to cycles $Z_1$ and $Z_2$, Poincar\'e dual to cohomology classes $\alpha_1$ and $\alpha_2$, and one boundary marked point which is the output, as shown in the centre of the diagram.  To the left, the interior marked points are brought together and bubble off, computing $\CO^0$ of the quantum product $\alpha_1 \qcup \alpha_2$.  To the right, meanwhile, the interior marked points are slid to the boundary of the disc, where they bubble off to give the Floer product of $\CO^0(\alpha_1)$ and $\CO^0(\alpha_2)$.  Counting boundary components of this cobordism proves that $\CO^0(\alpha_1 \qcup \alpha_2) = \CO^0(\alpha_1) \qcup \CO^0(\alpha_2)$ modulo bubbling, but the contributions from bubbled configurations are Floer coboundaries so vanish in $HF^*(\mathbf{L}, \mathbf{L}; \Lambda)$.

Now restrict to the case where $X$ is a toric variety and $L$ is our chosen toric fibre.  To be precise, we equip $L$ with the zero weak bounding cochain, an arbitrary orientation, the \emph{standard} relative spin structure as defined in \cite[Section 9]{ChoOh} or \cite[Definition 4.2.4]{SmDCS}, and the trivial local system.  The fact that zero \emph{is} a weak bounding cochain is a non-trivial result, but is beside the point of the present motivational discussion.  Choosing orientations just fixes $\Z/2$-gradings of Floer groups between different Lagrangians \cite{SeidelGr}, but gradings will be irrelevant to us; in any case, we shall only consider the Floer theory of $L$ with itself, and this is \emph{canonically} $\Z/2$-graded.  The choice of spin structure determines the signs of the disc counts, and using the standard one ensures that the basic discs $u_1, \dots, u_N$ all contribute positively \cite[Appendix B, Sections B.4 and B.5]{SmDCS}.

\begin{rmk}
There are different approaches to the construction of (virtual) fundamental chains on spaces of pseudoholomorphic curves.  The classical method is to deform the Cauchy--Riemann equation by varying the almost complex structure and incorporating perturbation terms, so that the moduli spaces become bona fide smooth manifolds.  In general, however, this is not possible and one must endow the (generally singular) moduli spaces with some extra data which encodes the cokernel of the defining operator, for example a Kuranishi structure, as introduced by Fukaya--Ono \cite{FukayaOno} (see \cite{PardonVFC} for another alternative).  This is the approach taken by Fukaya--Oh--Ohta--Ono in \cite{FOOObig} and subsequent work (including \cite{FOOOtorI,FOOOMSToric}), and hence is implicitly in the background of the present discussion.  Compared with the classical approach, it has the advantage that one may work with a fixed almost complex structure, at the expense of considering possibly-nodal stable maps as opposed to smooth ones, and until \cref{secFurtherDiscussion} we shall always use the standard integrable complex structure on our toric variety $X$.  For this reason we will drop the `pseudo' from `pseudoholomorphic'.
\end{rmk}

Fix a base point $p \in L$, and for a homotopy class $\beta \in \pi_2(X, L, p)$ consider the moduli space of holomorphic stable discs
\begin{multline}
\label{MbetaDefn}
\mathcal{M}_\beta = \{ u : (\Sigma, \pd \Sigma, *) \rightarrow (X, L, p) : (\Sigma, \pd \Sigma, *) \text{ is a possibly-nodal disc} \\ \text{with boundary marked point $*$, and $u$ is a holomorphic stable map representing } \beta \}.
\end{multline}
As discussed in \cref{sscToricBG}, this space has virtual dimension is $\mu(\beta) - 2$, so when $\beta$ has index $2$ the discs are rigid and we can count them by integrating $1$ against the virtual fundamental chain on the moduli space.  For such $\beta$ we define $W_\beta \in \Lambda_0[H_1(L; \Z)]$ by
\[
W_\beta = (\# \mathcal{M}_\beta) T^{\ip{\omega}{\beta}} \tau^{\pd \beta},
\]
where $\ip{\omega}{\beta}$ is the area of $\beta$ and $\tau$ is a formal variable whose exponents records the $H_1(L; \Z)$ class.  To avoid confusion we define $W_\beta$ to be $0$ when $\mu(\beta) \neq 2$.

By a positivity-of-intersections argument, the $W_\beta$ all lie in a certain subring of $\Lambda_0 [H_1(L; \Z)]$ and their sum converges in its completion $\LHD$ which we describe in \cref{sscCoeffRings}.  We define
\[
W = \sum_\beta W_\beta \text{\quad and \quad} W_j = \sum_\beta (D_j \cdot \beta) W_\beta
\]
for $j=1, \dots, N$.  Since the only discs which contribute have index $2$, and hence have intersection number $1$ with $D_1 + \dots + D_N$, we see that
\[
W = W_1 + \dots + W_N.
\]
This formalises the description of $W$ as a boundary-weighted disc count given in \cref{sscOverview}.

By similarly weighting counts of discs by the homology classes of their boundaries we can make sense of the self-Floer cohomology of $L$ over the coefficient ring $\LHD$.  This is discussed in the monotone case in \cite[Section 2.2.4]{BCQS}.  Explicitly, using Fukaya--Oh--Ohta--Ono's canonical de Rham model, in which the Floer complex is given by the de Rham cohomology
\[
CF^*_\mathrm{can} (L, L; \LHD) \coloneqq H^*_\mathrm{dR} (L; \R) \otimes_\R \LHD,
\]
the Floer differential $\diff$ applied to a class $b \in H^1(L; \R)$ gives $W^b$ times the unit $1_L$, where
\[
W^b = \sum_\beta \ip{b}{\pd \beta} W_\beta
\]
is the count of rigid discs through a generic point of $L$, weighted by area, boundary homology class, and the pairing of the boundary class with $b$.

This has the following important property:

\begin{lem}[{Corresponding to \cite[Proposition 2.4.16]{FOOOMSToric}}]
\label{lemHF}
The $\LHD$-subalgebra of
\[
HF^*(L, L; \LHD)
\]
generated by the unit $1_L$ is
\[
\LHD \Big/ \Big(\sum_{j=1}^N \nu_j W_j \Big).
\]
\end{lem}
\begin{proof}
Since the cohomology algebra of the torus is generated in degree $1$, by the Leibniz rule it is enough to show that the ideal generated by the $W^b$ is $(\sum_j \nu_j W_j)$.  Using the duality between the $[u_j]$ and the $H_j$, for any disc class $\beta$ we have
\[
\pd \beta = \pd \sum_{j=1}^N \ip{H_j}{\beta} [u_j] = \sum_{j=1}^N \ip{H_j}{\beta} \nu_j,
\]
and hence for all $b \in H^1(L; \R)$ we have
\[
W^b = \sum_\beta \ip{b}{\pd \beta} W_\beta = \sum_\beta \sum_{j=1}^N \ip{b}{\nu_j} \ip{H_j}{\beta} W_\beta = \sum_{j=1}^N \ip{b}{\nu_j} W_j,
\]
from which the result immediately follows.
\end{proof}

\begin{rmk}
\label{logDerivative}
Let $\gamma_1, \dots, \gamma_n$ be a basis of $H_1(L; \Z)$, $y_j = \tau^{\gamma_j}$ be the corresponding monomials in $\Lambda [H_1(L; \Z)]$, and $b_1, \dots, b_n$ be the dual basis of $H^1(L; \Z)$.  For any class $\beta$ with boundary $\pd \beta = \sum p_j \gamma_j$ we then have
\[
W_\beta = (\# \mathcal{M}_\beta) T^{\ip{\omega}{\beta}} y_1^{p_1}\dots y_n^{p_n}.
\]
Weighting by the pairing $\ip{b_j}{\pd \beta} = p_j$ therefore corresponds to applying the operator $y_j \pd / \pd y_j$, and we deduce that
\[
\diff b_j = y_j \frac{\pd W}{\pd y_j} \cdot 1_L
\]
for each $j$.  We can therefore rewrite the ideal appearing in \cref{lemHF} as $(y_j \pd W/\pd y_j)$.  The $y_j$ are coordinates on the space of local systems, and the operators $y_j \pd/\pd y_j = \pd / \pd \log y_j$ represent the derivatives after pulling back under the exponential map $H^1(L; \C) \rightarrow H^1(L; \C^*)$.
\end{rmk}

\begin{rmk}
The description of the superpotential in \cite{FOOOtorI} (where it is called $\mathfrak{PO}$) is superficially different from our description of $W$.  Rather than explicitly weighting disc counts by boundary homology class, Fukaya--Oh--Ohta--Ono introduce formal variables $x_1, \dots, x_n$ (the logarithms of the $y_j$ appearing in \cref{logDerivative}) and consider the sum
\[
\sum_{j=0}^\infty \m^\mathrm{can}_j(b, \dots, b),
\]
where $b = x_1b_1+\dots+x_nb_n$ and $\m^\mathrm{can}_*$ are the $A_\infty$-operations on $CF^*_\mathrm{can}(L, L; \Lambda_0)$.  This is effectively the restriction of the potential function of \cite{FOOObig}, which is defined on the space of weak bounding cochains on $L$, to the subspace comprising those cochains which correspond to local systems (coordinatised by their logarithms).  An open-string version of the divisor axiom \cite[Lemma 11.8]{FOOOtorI} shows that this sum is actually a function of the $y_j = e^{x_j}$, and that after making this substitution it coincides with definition we gave for $W$.
\end{rmk}

Just as we defined the Floer cohomology over $\LHD$, we can also define $\CO^0$ with $\LHD$ coefficients to obtain a $\Lambda_0$-algebra homomorphism
\begin{equation}
\label{CO0expression}
\CO^0 : QH^*(X; \Lambda_0) \rightarrow HF^*(L, L; \LHD).
\end{equation}
The left-hand side is spanned by the Poincar\'e duals of cycles on $X$ which are setwise-invariant under the torus-action, and one therefore expects the image to be spanned by torus-invariant classes on $L$.  The only invariant cycle on $L$ is the fundamental cycle, so $\CO^0$ should map to the subalgebra of the right-hand side generated by $1_L$, and we obtain, heuristically at least, a $\Lambda$-algebra homomorphism
\begin{equation}
\label{CO0ks}
QH^*(X; \Lambda_0) \rightarrow \LHD \Big/ \Big(\sum_{j=1}^N \nu_j W_j \Big).
\end{equation}
The Kodaira--Spencer map $\ks$ is defined by formalising this.

Of course, in order to do this we do not need to define the Fukaya category or the full $\CO$ map: \eqref{CO0ks} can be defined directly by counting discs with an interior marked point constrained to the Poincar\'e dual of the input class, and a boundary marked point constrained to $p \in L$ (linear dual to $1_L$), weighted by area and boundary homology class.  Taking $\alpha$ to be $H_j$, represented by the Poincar\'e dual cycle $D_j$, we see directly from the definition of $\CO^0$ that $\CO^0(H_j)$ is given by the same expression as $W_j$.  The fact that the image of \eqref{CO0expression} is contained in the subalgebra generated by the unit corresponds to the fact that Fukaya--Oh--Ohta--Ono construct the relevant Kuranishi structures \emph{torus-equivariantly}, so that the moduli spaces have to sweep a multiple of the fundamental class when evaluating to $L$.  This concludes our discussion of steps \eqref{itmSuperpot} and \eqref{itmks}.

\subsection{The cone $C$}
\label{sscCone}

In \cref{sscCoeffRings} we will define the ring $\LHD$ that appears in our main results.  So far we have discussed this as a completion of a subring of $\Lambda_0 [H_1(L; \Z)]$, where $L$ is our chosen toric fibre, but it has an alternative description as a completion of the monoid ring $\C[\Gamma_\R]$ associated to a submonoid $\Gamma_\R$ of $\R \oplus H_1(L; \Z)$, noting that $\Lambda_0$ is itself a completion of the monoid ring $\C[\R_{\geq 0}]$.  The monoid perspective, and in particular the related geometry, is more useful for our purposes and is therefore the one we shall adopt from now on.

Our first step is to define a convex cone $C$ in $\R \oplus H_1(L; \R)$ which will contain $\Gamma_\R$.  The slogan one should keep in mind throughout is that $\Gamma_\R$ is the smallest monoid which can record areas and boundaries of holomorphic stable discs, and $C$ is the smallest cone containing it.  Recall that the areas and boundaries of the basic disc classes $[u_j]$ are given respectively by the distances $\lambda_j$ from the facets of $\Delta$ to the origin and by the normals $\nu_j$ to these facets (as in \eqref{Polytope}).

\begin{defn}
The cone $C$ is defined to be the subset of $\R \oplus H_1(L; \R)$ given by $\R_{\geq 0}$-linear combinations of the vectors $(\lambda_j, \nu_j)$ and the vector $(1, 0)$.
\end{defn}

It will be helpful to have an alternative formulation, as follows.  For each vertex $v$ of $\Delta$ let $\theta_v : \R \oplus H_1(L; \R) \rightarrow \R$ be the linear functional defined by
\[
\theta_v(\lambda, \nu) = \lambda + \ip{v}{\nu}.
\]
Let $C_\theta$ be the subset of $\R \oplus H_1(L; \R)$ on which all $\theta_v$ are non-negative, and let $C_{H_1} \subset H_1(L; \R)$ be the set of $\R_{\geq 0}$-linear combinations of $\nu_1, \dots, \nu_N$.

\begin{lem}
\label{CReformulation}
$C$ is given by the intersection of $C_\theta$ with $\R \times C_{H_1}$.
\end{lem}
\begin{proof}
For each point $x$ of $\Delta$ we have $\lambda_j + \ip{\nu_j}{x} \geq 0$ for all $j$.  In particular, this holds for each vertex $v$, which means that every $(\lambda_j, \nu_j)$ lies in $C_\theta$.  Clearly $(1, 0)$ also lies in $C_\theta$, so every non-negative combination of these points is in $C_\theta$ too.  Any such non-negative combination also obviously lies in $\R \times C_{H_1}$ so we see that
\[
C \subset C_\theta \cap (\R \times C_{H_1}).
\]

Conversely suppose that $c = (\lambda, \nu)$ is a point in $C_\theta \cap (\R \times C_{H_1})$.  We claim that the functional $\nu : H^1(L; \R) \rightarrow \R$ is bounded below on $\Delta$ (we are thinking of $H^1$ as the dual space to $H_1$ and no longer explicitly notating the pairing between these spaces).  Given this, it attains its minimum value at some vertex $v$, and by relabelling we may assume that the facets meeting at $v$ are $F_1, \dots, F_n$.  Since $\nu_1, \dots, \nu_n$ form a basis for $H_1(L; \Z)$ we can write $\nu$ as $t_1\nu_1 + \dots + t_n\nu_n$ for some (unique, possibly negative) real numbers $t_1, \dots, t_n$, and then write $c$ as
\[
(s, 0) + t_1(\lambda_1, \nu_1) + \dots + t_n(\lambda_n, \nu_n)
\]
for some (again, unique but possibly negative) real number $s$.  We are done if we can show that $s$ and the $t_j$ are in fact non-negative.

To see this is indeed the case, note that since $c$ lies in $C_\theta$ we have $\theta_v(c) \geq 0$, and since $\nu$ attains its minimum value at $v$ it is non-decreasing along the edges of $\Delta$ which emanate from there.  The former (combined with the fact that for $j=1, \dots, n$ we have
\[
\theta_v(\lambda_j, \nu_j) = \lambda_j + \ip{\nu_j}{v} = 0
\]
since $v$ lies in the facet $F_j$) tells us that $s$ is non-negative, whilst the latter says precisely that the $t_j$ are non-negative.

It is therefore left to show that $\nu$ is bounded below on $\Delta$, and this follows from the fact that $c$ is in $\R \times C_{H_1}$ (which we have not yet used): we can write $\nu$ as $t_1 \nu_1 + \dots + t_N \nu_N$ for some non-negative $t_1, \dots, t_N$, and each $\nu_j$ is bounded below by $-\lambda_j$ on $\Delta$, so $\nu$ is bounded below by $-(t_1\nu_1 + \dots + t_N\nu_N)$.
\end{proof}

The next result is not strictly necessary but is helpful for thinking about $C$ geometrically:

\begin{lem}
\label{CWhenDeltaCompact}
The following are equivalent:
\begin{enumerate}[(a)]
\item\label{Citm1} $\Delta$ is compact.
\item\label{Citm2} $C_{H_1}$ is the whole of $H_1(L; \R)$.
\item\label{Citm3} $(1, 0)$ is in the interior of $C$ (which implies that it is redundant as a generator of $C$).
\end{enumerate}
\end{lem}
\begin{proof}
If $\Delta$ is compact then the functional $\nu_1 : H^1(L; \R) \rightarrow \R$ is bounded above on $\Delta$ and attains its maximum value at some vertex $v$.  By relabelling we may assume that the normals to the facets meeting at $v$ are $F_2, \dots, F_{n+1}$, and since $\nu_2, \dots, \nu_{n+1}$ form a basis for $H_1(L; \Z)$ we can write $\nu_1$ as $t_2 \nu_2 + \dots + t_{n+1}\nu_{n+1}$ for some integers $t_2, \dots, t_{n+1}$.  Because $\nu_1$ attains its maximum value at $v$, the $t_j$ are non-positive, and hence
\[
-\nu_1 = - t_2 \nu_2 - \dots - t_{n+1} \nu_{n+1}
\]
is an expression for $-\nu_1$ as an element of $C_{H_1}$.  Repeating this argument with each $\nu_j$ in place of $\nu_1$ in turn, we see that $C_{H_1}$ contains $\pm \nu_j$ for all $j$ and thus is the whole of $H_1(L; \R)$, proving that \eqref{Citm1} implies \eqref{Citm2}.

Conversely, if $\Delta$ is non-compact then it contains some ray $\R_{\geq 0}x$.  This forces $\ip{x}{\nu_j}$ to be non-negative for all $j$, and hence $C_{H_1}$ is contained in the half-space $\ip{x}{\cdot} \geq 0$.  Thus \eqref{Citm2} implies \eqref{Citm1}.

Finally, each $\theta_v$ evaluates to $1$ on $(1, 0)$, so it lies in the interior of $C_\theta$.  This means that \eqref{Citm3} is equivalent to $0$ lying in the interior of $C_{H_1}$, and since the latter is a cone this is equivalent to \eqref{Citm2}.
\end{proof}

Using the reformulation \cref{CReformulation}, we can make:

\begin{defn}
We define the \emph{bottom boundary} of $C$, denoted $\pd C$, to be
\[
\pd C = \{c \in C : \theta_v(c) = 0 \text{ for some vertex } v\}.
\]
Equivalently, it is the intersection of $C$ with the boundary of $C_\theta$.  If $\Delta$ is compact then by \cref{CWhenDeltaCompact} this is the usual boundary of $C$, but in general we take the usual boundary and delete the interiors of the faces which contain the `vertical' direction $\R_{\geq 0} \oplus 0$.  The \emph{height} of an element $c = (\lambda, \nu)$ of $C$ is defined to be
\[
h(c) = \min_v \theta_v(c).
\]
This is the unique non-negative real number such that $c$ is contained in $\pd C + (h(c), 0)$.
\end{defn}

\Cref{figCone} illustrates $C$, $\pd C$ and the subset of $C$ comprising elements of height at least $h$ for $\C\P^1$ and $\C$, with moment polyhedra $\{-\lambda_1 \leq x \leq \lambda_2\}$ and $\{-\lambda_1 \leq x\}$ respectively.  We have chosen $0 < \lambda_1 < \lambda_2$, which means that the upward-sloping part of $\pd C$ for $\C\P^1$ is steeper than the downward-sloping part.

\begin{figure}[ht]
\centering
\begin{tikzpicture}
\def\l{0.8}
\def\ltot{1.3}
\def\xmax{3}
\def\ymin{-0.1}
\def\ymax{4.5}

\begin{scope}
\begin{scope}
\clip (-\xmax, \ymin) rectangle (\xmax, \ymax);
\draw[draw=orange, line width=0.4mm, dashed, fill=orange, fill opacity=0.2] (0, 0) -- ($8*(1, \l)$) -- ++(0, 5) -- ++(-16, 0) -- ($8*(-1, \ltot-\l)$) -- cycle;
\begin{scope}[shift={(0, 2)}]
\draw (0, 0) node[anchor=east]{\small $h$};
\fill[orange, opacity=0.2] (0, 0) -- ($8*(1, \l)$) -- ++(0, 5) -- ++(-16, 0) -- ($8*(-1, \ltot-\l)$) -- cycle;
\end{scope}
\end{scope}
\draw[->] (-\xmax-0.25, 0) -- (\xmax+0.25, 0);
\draw[->] (0, \ymin-0.4) -- (0, \ymax+0.4);
\draw (0, \ymax+0.3) node[anchor=west]{\small $\R$};
\draw (\xmax+0.25, 0) node[anchor=south]{\small $H_1(L; \R)$};
\end{scope}

\def\xmin{-0.1}
\def\xmax{5}

\begin{scope}[shift={(6.2, 0)}]
\begin{scope}
\clip (-\xmax, \ymin) rectangle (\xmax, \ymax);
\fill[orange, opacity=0.2] (0, 0) -- ($8*(1, \l)$) -- ++(0, 5) -- ++(-8, 0) -- cycle;
\draw[draw=orange, line width=0.4mm, dashed] (0, 0) -- ($8*(1, \l)$);
\begin{scope}[shift={(0, 2)}]
\draw (0, 0) node[anchor=east]{\small $h$};
\fill[orange, opacity=0.2] (0, 0) -- ($8*(1, \l)$) -- ++(0, 5) -- ++(-8, 0) -- cycle;
\end{scope}
\end{scope}
\draw[->] (\xmin-0.4, 0) -- (\xmax+0.25, 0);
\draw[->] (0, \ymin-0.4) -- (0, \ymax+0.4);
\draw (0, \ymax+0.3) node[anchor=west]{\small $\R$};
\draw (\xmax+0.25, 0) node[anchor=south]{\small $H_1(L; \R)$};
\end{scope}

\end{tikzpicture}
\caption{The cone $C$ (shaded), bottom boundary $\pd C$ (dashed), and subset of $C$ of height at least $h$ (dark shaded) for $\C\P^1$ (left) and $\C$ (right).\label{figCone}}
\end{figure}
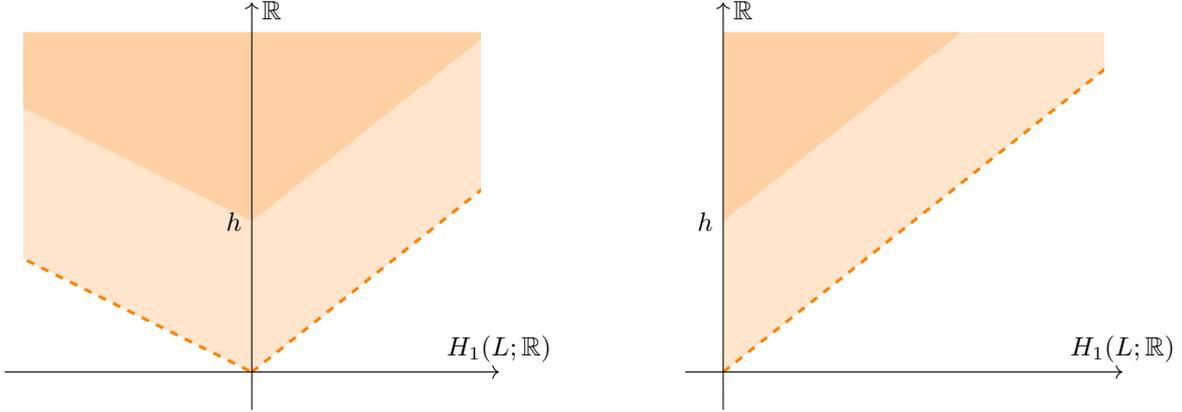

We have the following important connection between height and faces of $\Delta$:

\begin{lem}[Height--intersection lemma]
\label{HILemma}
For non-negative real numbers $s, t_1, \dots, t_N$ the sum
\begin{equation}
\label{cSum}
s(1, 0) + t_1(\lambda_1, \nu_1) + \dots + t_N(\lambda_N, \nu_N)
\end{equation}
has height at least $s$, with equality if and only if the intersection
\[
\bigcap_{j : t_j > 0} F_j
\]
is non-empty---we shall call sums with the latter property \emph{intersecting}.  Every element of $C$ has a unique representation as an intersecting sum.
\end{lem}
\begin{proof}
Clearly \eqref{cSum} has height at least $s$, with equality if and only if there exists $v$ with
\[
\theta_v (t_1 (\lambda_1, \nu_1) + \dots + t_N (\lambda_N, \nu_N) ) = 0.
\]
Expanding out the left-hand side as $\sum_j t_j \theta_v(\lambda_j, \nu_j)$ and using the fact that each $\theta_v(\lambda_j, \nu_j)$ is non-negative with equality if and only if $v$ lies in $F_j$, we see that equality occurs if and only if there exists a $v$ which is contained in all $F_j$ for which $t_j$ is non-zero.  In other words, if and only if it is intersecting.

Now let $c$ be an arbitrary element of $C$.  There exists some vertex $v$ of $\Delta$ such that $\theta_v(c) = h(c)$, and without loss of generality we may assume that the facets meeting at $v$ are $F_1, \dots, F_n$.  The element $c' = c - h(c)(1, 0)$ lies in $C$ by \cref{CReformulation}, so it can be written as an $\R_{\geq 0}$-linear combination of the form \eqref{cSum}, and since it satisfies $\theta_v(c') = 0$ we must have $s = t_{n+1} = \dots = t_N = 0$.  Thus
\[
c = h(c)(1, 0) + t_1(\lambda_1, \nu_1) + \dots + t_n(\lambda_n, \nu_n)
\]
(note the upper limit on this sum) is an expression for $c$ as an intersecting sum.

Suppose that
\[
s'(1, 0) + t'_1(\lambda_1, \nu_1) + \dots + t'_N(\lambda_N, \nu_N)
\]
is another intersecting sum expression for $c$ (again, note the upper limit).  Using once more the fact that $\theta_v(c) = h(c)$, we must have $s = h(c)$ and $t'_j = 0$ for $j=n+1, \dots, N$.  We deduce that
\[
(t_1-t_1')(\lambda_1, \nu_1) + \dots + (t_n-t_n') (\lambda_n, \nu_n) = 0,
\]
and because the normals $\nu_1, \dots, \nu_n$ to the facets meeting at $v$ are linearly independent, we conclude that $t_j=t'_j$ for all $j$, giving the claimed uniqueness.
\end{proof}

\subsection{The ring $\LHD$}
\label{sscCoeffRings}

Using the cone $C$ we can now define the main algebraic objects:

\begin{defn}
$\Gamma_\R$ is the intersection of $C$ with $\R \oplus H_1(L; \Z)$---note the $\Z$ coefficients on the homology group.  It carries a natural decreasing $\R_{\geq 0}$-filtration $F^* \Gamma_\R$ in which $F^\lambda \Gamma_\R$ is given by $\Gamma_\R \cap (C + (\lambda, 0))$, i.e.~the set of elements of height at least $\lambda$.  This height filtration is inherited by the monoid ring $\C[\Gamma_\R]$, and $\LHD$ is defined to be the completion.  Note that $\LHD$ is naturally a $\Lambda_0$-algebra, and the topology it inherits from the filtration on $\C[\Gamma_\R]$ is equivalent to the $T$-adic topology.
\end{defn}

We will use the \emph{height} of a monomial in $\C[\Gamma_\R]$ to mean simply the height of the corresponding element of $\Gamma_\R$.  An element of $\LHD$ is then a possibly-infinite sum of monomials whose heights tend to infinity, and we define the height of this element to be the minimum of the heights of the monomial summands.

When discussing monoid rings associated to submonoids of $\R \oplus H_1(L; \Z)$, as well as the corresponding completions, we shall encode the $\R$-component as the exponent of a formal variable $T$, just as in $\Lambda_0$.  We shall similarly encode the $H_1(L; \Z)$-component as the exponent of a formal variable $\tau$, as in \cref{sscksSketch}.

\begin{defn}
The monomials $\lv_1, \dots, \lv_N$ (appearing in \cref{sscOverview} and the main results) are defined by
\[
\lv_j = T^{\lambda_j} \tau^{\nu_j}.
\]
These are precisely the monomials corresponding to the areas and boundaries of the basic discs.
\end{defn}

The $\lv_j$, together with the powers of $T$, actually generate $\C[\Gamma_\R]$ as an algebra:

\begin{lem}
\label{GammaRSpan}
For any $c$ in $\Gamma_\R$, in the unique expression for $c$ as an intersecting sum the coefficients of the $(\lambda_j, \nu_j)$ are all integral.  Thus $\Gamma_\R$ is the submonoid of $\R \oplus H_1(L; \Z)$ generated by
\[
(\lambda_1, \nu_1), \dots, (\lambda_N, \nu_N) \text{ and } \R_{\geq 0} \oplus 0.
\]
\end{lem}
\begin{proof}
To prove the second statement, note that $\Gamma_\R$ obviously contains the claimed monoid, and the reverse inclusion follows from the first statement.  We are therefore left to prove the first statement, so fix an arbitrary $c$ in $\Gamma_\R$ and let its unique expression as an intersecting sum be
\[
c=h(c)(1, 0) + t_1(\lambda_1, \nu_1) + \dots + t_n(\lambda_n, \nu_n),
\]
where $F_1, \dots, F_n$ intersect at a vertex $v$.  Since $t_1 \nu_1 + \dots + t_n \nu_n$ lies in $H_1(L; \Z)$, and $\nu_1, \dots, \nu_n$ form a $\Z$-basis for this group by the Delzant condition, the $t_j$ must all be integral.
\end{proof}

We now need to check that our disc counts all lie in $\LHD$ as just defined, i.e.~that each such count is a sum over disc classes $\beta_1, \beta_2, \dots$ in $H_2(X, L; \Z)$ which satisfy:
\begin{enumerate}[(a)]
\item\label{itmGamma} $(\ip{\omega}{\beta_j}, \pd \beta_j)$ lies in $\Gamma_\R$ for all $j$.
\item\label{itmHeights} The heights $h(\ip{\omega}{\beta_j}, \pd \beta_j)$ tend to infinity as $j \rightarrow \infty$.
\end{enumerate}

The first of these properties follows from:

\begin{lem}[{Corresponding to \cite[Theorem 11.1(5)]{FOOOtorI}}]
\label{DiscsInGamma}
The relative homology class $\beta$ of any holomorphic stable disc $u$ in $X$, bounded by $L$, satisfies $(\ip{\omega}{\beta}, \pd \beta) \in \Gamma_\R$.
\end{lem}
\begin{proof}
Any such stable disc can be decomposed into smooth holomorphic disc and sphere components, and since the latter have non-negative area and zero boundary it suffices to show that the class of each disc component is a sum of basic disc classes.  By \cref{DiscDivisorPairing}, it is thus enough to show that each disc component pairs non-negatively with each toric divisor, and this follows from positivity of intersections (this may fail for sphere components since they may be contained within one of the divisors).
\end{proof}

To prove \eqref{itmHeights} we need to check that for any fixed $\lambda > 0$ each disc count only involves finitely many terms of height at most $\lambda$.  In order to show this, the important observation is that since we always count \emph{rigid} discs, we only need consider sums over discs of fixed Maslov index since this determines the dimensions of the relevant moduli spaces (this fails in the presence of bulk deformations; see \cref{sscBulk}).  With this in mind, convergence follows from:

\begin{lem}[{Corresponding to \cite[Lemma 2.4.14]{FOOOMSToric}}]
\label{Convergence}
For any integer $k$, and any positive real number $\lambda$, the set
\begin{multline*}
\mathcal{B} = \{ \beta \in H_2(X, L; \Z) : \mu(\beta) = 2k \text{ and $\beta$ is represented by a holomorphic stable disc}
\\ \text{whose sphere components have total area at most $\lambda$}\}
\end{multline*}
is finite.  In particular, the subset comprising classes of holomorphic stable discs which correspond to monomials of height at most $\lambda$ is finite.
\end{lem}
\begin{proof}
For each $\beta$ in $\mathcal{B}$ fix a holomorphic stable disc $u_\beta$ representing it, and let $\beta_D$ and $\beta_S$ be the total homology classes of the disc and sphere components of $u_\beta$ respectively.  By definition we may choose $u_\beta$ such that $\ip{\omega}{\beta_S} \leq \lambda$.

Applying Gromov compactness to the sphere components of the $u_\beta$ we deduce that as $\beta$ varies the class $\beta_S$ takes only finitely many values.  In particular, there exists an integer $l$ such that $\mu(\beta_S) \geq -2l$ for all $\beta \in \mathcal{B}$, and hence $\mu(\beta_D) \leq 2(k+l)$.  We also know that each $\beta_D$ is of the form $\sum_j m_j [u_j]$ for some non-negative integers $m_j$, with
\[
\mu(\beta_D) = 2 \sum_{j=1}^N m_j \quad \text{and} \quad \ip{\omega}{\beta_D} = \sum_{j=1}^N m_j\lambda_j,
\]
so if $\lambda_\mathrm{max}$ denotes the maximum of the $\lambda_j$ then we have
\[
\ip{\omega}{\beta_D} \leq \lambda_\mathrm{max}\frac{\mu(\beta_D)}{2} \leq \lambda_\mathrm{max}(k+l).
\]

For all $\beta$ in $\mathcal{B}$ we thus have
\[
\ip{\omega}{\beta} \leq \lambda_\mathrm{max}(k+l) + \lambda.
\]
The classes in $\mathcal{B}$ therefore have bounded area, so we can apply Gromov compactness again---this time to the whole stable discs $u_\beta$---to deduce the result.
\end{proof}

Thus \eqref{itmGamma} and \eqref{itmHeights} are proved, and our discussion of \eqref{itmRings} is complete.

\subsection{Leading-order terms}
\label{sscLeading}

We now discuss the final point in the construction of $\ks$, namely step \eqref{itmLeading}.  This is the computation of the leading-order terms of the summands $W_j$ of the superpotential:

\begin{lem}[{Corresponding to \cite[Theorem 4.6]{FOOOtorI}}]
For each $j$ we have
\[
W_j = \lv_j \mod \Lambda_+.
\]
\end{lem}
\begin{proof}
Suppose that $u$ is a holomorphic stable disc contributing to the leading-order term of $W_j$, and let $\beta_D$ and $\beta_S$ be the classes of its disc and sphere components respectively.  We then have
\[
(\ip{\omega}{u}, [\pd u]) = (\ip{\omega}{\beta_D}, \pd \beta_D) + (\ip{\omega}{\beta_S}, 0),
\]
and we know from the proof of \cref{DiscsInGamma} that the first term on the right-hand side lies in $\Gamma_\R$.  The fact that $u$ contributes to the leading-order term means that the left-hand side lies in $\Gamma_\R \setminus F^{>0}\Gamma_\R$, so we must have $\ip{\omega}{\beta_S} = 0$, and hence $u$ has no sphere components.

Now, in order to be rigid, each disc contributing to $W$ has index $2$, which tells us that the intersection number of $\beta_D$ with the sum of the toric divisors is $1$.  By positivity of intersections we deduce that exactly one component of $u$ meets exactly one of the toric divisors, transversely, and there are no other intersections.  It is then a standard piece of toric geometry (for example by using a toric chart in which the complement of the divisors is identified with $(\C^*)^n$, and the complement of all-but-one of the divisors is identified with $\C \times (\C^*)^{n-1}$, with $L$ given by the product of the unit circles in each case) to see that $u$ has a single non-constant component, which is precisely a basic disc $u_j$ or one of its translates under the torus-action.  Compare with \cite[Lemma 11.1(2)--(3)]{FOOOtorI}.

We thus have for each $j$ that
\[
W_j = \sum_{k=1}^N  \ip{H_j}{[u_k]} W_{[u_k]} = W_{[u_j]} = (\#\mathcal{M}_{[u_j]}) \lv_j \mod \Lambda_+ ,
\]
where $\mathcal{M}_{[u_j]}$ is the moduli space of holomorphic stable discs in class $\beta$ with a boundary marked point constrained to $p \in L$, as defined in \eqref{MbetaDefn}.  We just saw that before applying this constraint the moduli space comprises the free torus orbit of the disc $u_j$.  Fukaya--Oh--Ohta--Ono show \cite[Lemmas 11.1(3) and 11.2(5)]{FOOOtorI} that this unconstrained moduli space is regular, that the evaluation map from the boundary marked point to $L$ is an orientation-preserving diffeomorphism (using the standard spin structure to orient the domain), and that in their construction of the required Kuranishi structures it can be left unperturbed.  This implies that $\#\mathcal{M}_{[u_j]} = 1$, proving the lemma.
\end{proof}

\section{The proof of \cref{Theorem1}}
\label{secProof}

\subsection{The Stanley--Reisner ring}
\label{sscSR}

The heart of the proof of \cref{Theorem1} is to relate $\LHD$ to the Stanley--Reisner ring and study that instead, and this is the subject of the present subsection and the next.  Recall from \cref{rmkGroundField} that \cref{propFOOOks} relies on the ground field containing $\R$, so for simplicity we phrased our results in terms of the Novikov ring over $\C$.  However, \cref{Theorem1} itself holds over any ground field, so \textbf{for the rest of \cref{secProof} we replace $\C$ with an arbitrary ground field $\F$}.  For example, the Stanley--Reisner and Novikov rings are now given by
\[
SR(X) = \F[Z_1, \dots, Z_N] / ( Z_{j_1} \dots Z_{j_k} : F_{j_1} \cap \dots \cap F_{j_k} = \emptyset )
\]
and
\[
\Lambda_0 = \Big\{\sum_{j=1}^\infty a_jT^{l_j} : a_j \in \F \text{ and } l_j \in \R_{\geq 0} \text{ with } l_j \rightarrow \infty \text{ as } j \rightarrow \infty\Big\}.
\]

We begin with:

\begin{lem}
\label{SRRing}
The $\F$-algebra map $\phi : \F[Z_1, \dots, Z_N] \rightarrow \LHD$ defined by $Z_j \mapsto \lv_j$ induces an isomorphism
\[
\overline{\phi} : SR(X) \rightarrow \LHD / (\Lambda_+ \cdot \LHD).
\]
\end{lem}
\begin{proof}
First note that the ideal defining the Stanley--Reisner ring has an $\F$-basis given by monomials $Z_1^{m_1} \dots Z_N^{m_N}$ (with the $m_j$ non-negative integers) such that
\begin{equation}
\label{mjIntersection}
\bigcap_{j : m_j > 0} F_j
\end{equation}
is empty, so the ring itself has a basis given by those monomials for which the intersection is \emph{non-empty}.

On the other hand, directly from the definition of $\LHD$ we see that $\LHD / (\Lambda_+ \cdot \LHD)$ is the quotient of the monoid ring $\F[\Gamma_\R]$ by the span of the monomials of strictly positive height.  It therefore has a basis given by those monomials of height zero.  (Warning: there are two different types of `monomial' in play.  In the previous paragraph we were talking about formal expressions $Z_1^{m_1} \dots Z_N^{m_N}$, whilst we are now talking about elements $T^\lambda \tau^\nu$ of $\F[\Gamma_\R]$ corresponding to a single element of $\Gamma_\R$.)

By the height--intersection lemma (\cref{HILemma}), each generator of the Stanley--Reisner ideal is sent to an element of $\F[\Gamma_\R]$ of strictly positive height, so $\phi$ does indeed induce a map $\overline{\phi}$ as claimed.  Moreover, combining height--intersection with \cref{GammaRSpan} we see that every monomial in $\F[\Gamma_\R]$ of height zero has a unique representation of the form $\lv_1^{m_1} \dots \lv_N^{m_N}$ such that \eqref{mjIntersection} is non-empty.  This means that $\overline{\phi}$ induces a bijection between an $\F$-basis of its domain and an $\F$-basis of its codomain, and hence is an isomorphism.
\end{proof}

We can now prove the Stanley--Reisner presentation \eqref{SRPresentation}.  Although this is a standard fact, the argument is somewhat novel in that it uses leading-order quantum topology to prove a statement in classical topology (the usual proof involves identifying $SR(X)$ with the torus-equivariant cohomology of $X$ and then setting the equivariant parameters to zero):

\begin{prop}
\label{propSRPresentation}
The $\F$-algebra map $\psi : \F[Z_1, \dots, Z_N] \rightarrow H^*(X; \F)$ defined by $Z_j \mapsto H_j$ induces an isomorphism
\[
\overline{\psi} : SR(X) \Big/ \Big( \sum_{j=1}^N \nu_j Z_j \Big) \rightarrow H^*(X; \F).
\]
\end{prop}
\begin{proof}
Since the $H_j$ generate the cohomology algebra, the map $\psi$ is surjective.  And since the cup product is Poincar\'e dual to the intersection product, each monomial $H_1^{m_1}\dots H_N^{m_N}$ for which \eqref{mjIntersection} is empty vanishes in $H^*(X; \F)$.  Hence $\psi$ factors through $SR(X)$.

Next we need to check that $\psi$ annihilates the $n$-dimensional system of expressions $\sum_j \nu_j Z_j$.  To do this recall that we have dual short exact sequences
\begin{gather*}
0 \rightarrow H_2(X; \Z) \rightarrow H_2(X, L; \Z) \rightarrow H_1(L; \Z) \rightarrow 0
\\ 0 \leftarrow H^2(X; \Z) \leftarrow H^2(X, L; \Z) \leftarrow H^1(L; \Z) \leftarrow 0
\end{gather*}
and the middle groups have dual bases $[u_1], \dots, [u_N]$ and $H_1, \dots, H_N$ respectively.  In the top sequence the image of $[u_j]$ in $H_1(L; \Z)$ is $\nu_j$, so the kernel of the map $H^2(X, L; \Z) \rightarrow H^2(X; \Z)$ is spanned by $\sum_j \nu_j H_j$, and hence $\psi$ does indeed send $\sum_j \nu_j Z_j$ to zero.

Finally we need to check that $\overline{\psi}$ is injective, and this is where the quantum topology comes in.  By \cref{propFOOOks} we have a $\Lambda_0$-algebra homomorphism $\ks : QH^*(X; \Lambda_0) \rightarrow \LHD / (\sum_j \nu_j W_j)$ whose reduction modulo $\Lambda_+$ is an $\F$-algebra homomorphism
\[
\overline{\ks} : H^*(X; \F) \rightarrow \LHD \Big/ \Big(\Lambda_+ \cdot \LHD + \Big(\sum_{j=1}^N \nu_j W_j\Big) \Big)
\]
which sends $H_j$ to $\lv_j$.  By \cref{SRRing} the codomain can be identified with $SR(X) / (\sum_j \nu_j Z_j)$ and under this identification $\overline{\ks}$ constitutes a left inverse to $\overline{\psi}$.  This means that $\overline{\psi}$ is injective, and therefore an isomorphism.
\end{proof}

\begin{rmk}
\label{rmkSRObjections}
One might object to this argument, and in particular to its reliance on \cref{propFOOOks}, for (at least) two reasons: it seems to require the ground field $\F$ to contain $\R$, and also to employ the difficult torus-equivariant Kuranishi machinery of Fukaya--Oh--Ohta--Ono.  But since we only need to construct $\overline{\ks}$, not its lift $\ks$, neither of these is in fact the case.  The point is that when working modulo $\Lambda_+$ one can ignore all sphere bubbling, since stable discs with sphere bubbles contribute monomials of positive height.  This just leaves trees of smooth holomorphic discs, and all moduli spaces of such configurations are already regular because of torus-invariance of the standard complex structure (\cite[Lemma 3.2]{EL1} deals with the case of a single disc, whilst the torus action on the disc moduli spaces ensures that the fibre products defining trees are transverse).  Such configurations can therefore be counted over $\Z$---so we can actually replace $\F$ by $\Z$---and without recourse to virtual perturbation theory, and the map $\overline{\ks}$ can be defined using the pearl model of Biran--Cornea \cite{BCQS} (but fixing the complex structure and varying the Morse data to achieve transversality, as in \cite[Appendix A]{Sm}).  One needs to use the fact that $H^*(X; \Z)$ is generated as a ring by $H^2(X, L; \Z)$ to see that this map lands in the subalgebra of $HF^*(L, L)$ generated by the unit, in the sense discussed at the end of \cref{sscksSketch}.
\end{rmk}

\subsection{The main lemmas}
\label{sscMainLemmas}

Having established the connection between $SR(X)$ and both $\LHD$ and $H^*(X; F)$, we next prove the two key technical lemmas of the paper:

\begin{lem}
\label{SRisCM}
The Stanley--Reisner ring $SR(X)$ is Cohen--Macaulay.
\end{lem}

\begin{lem}
\label{RegSeq}
The components of $\sum_j \nu_j Z_j$ (with respect to any basis of $H_1(L; \Z)$) form a regular sequence in $SR(X)$.
\end{lem}

The former is only used a step towards the latter, and really needs the fact that the ground field $\F$ is field, not an arbitrary ring (which was irrelevant in the previous subsection).

\begin{proof}[Proof of \cref{SRisCM}]
Let $\Sigma$ be the simplicial complex whose vertices are the facets $F_1, \dots, F_N$ of $\Delta$, and in which a collection of vertices spans a face of $\Sigma$ if and only if the intersection of the corresponding facets is non-empty.  The Stanley--Reisner ring of a simplicial complex is defined to be the polynomial ring on its vertices, modulo the ideal spanned by products of vertices which are not contained in a single face, so by construction the Stanley--Reisner ring of $\Sigma$ is exactly $SR(X)$.

Reisner \cite[Theorem 1]{ReisnerCMQuotients} showed that such a ring is Cohen--Macaulay if and only if the reduced homology of the geometric realisation $|\Sigma|$, and of the link $L$ of each face, is zero in all degrees below $\dim \Sigma$, respectively $\dim L$.  Munkres \cite[Theorem 3.1]{Munkres} showed that this is equivalent to the purely topological condition
\[
\widetilde{H}_i(|\Sigma|; \F) = H_i(|\Sigma|, |\Sigma| \setminus \{p\}; \F) = 0
\]
for all $p$ in $|\Sigma|$ and all $i < \dim \Sigma$, and in particular \cite[Corollary 3.2]{Munkres} that the ring is Cohen--Macaulay whenever $|\Sigma|$ is a topological manifold with or without boundary whose reduced homology over $\F$ vanishes in degrees less than its dimension.  It therefore suffices to prove the following claim: $|\Sigma|$ is homeomorphic to either the sphere $S^{n-1}$ (when $X$ is compact) or the ball $B^{n-1}$ (when $X$ is non-compact).

Let $\Delta^*$ be the slice $C \cap (\{1\} \times H_1(L; \R))$ through $C$, which is precisely the space of $\R_{\geq 0}$-linear combinations
\begin{equation}
\label{DeltaDual}
s(1, 0) + \frac{t_1}{\lambda_1} (\lambda_1, \nu_1) + \dots + \frac{t_N}{\lambda_N} (\lambda_N, \nu_N)
\end{equation}
satisfying $s+t_1+\dots+t_N = 1$, and let $\pd \Delta^*$ denote its intersection with $\pd C$.  Equivalently, forgetting the $\{1\}$ factor, $\Delta^*$ is the convex hull in $H_1(L; \R)$ of the point $0$ and the rescaled normals $\nu_j/\lambda_j$, and $\pd \Delta^*$ is its ordinary boundary but with the interiors of the faces containing $0 \in H_1(L; \R)$ deleted.  By the latter description, $\pd \Delta^*$ is homeomorphic to $S^{n-1}$ if $0$ lies in the interior of $\Delta^*$---which is equivalent by \cref{CWhenDeltaCompact} to $\Delta$ being compact---or $B^{n-1}$ otherwise, so our claim is reduced to the following: $|\Sigma|$ is homeomorphic to $\pd \Delta^*$.  \Cref{figDeltaDual} shows $\Delta^*$ and $\pd \Delta^*$ for $\C\P^1$ and $\C$; note that $\pd \Delta^*$ is $S^0$ and $B^0$ in the two cases respectively.

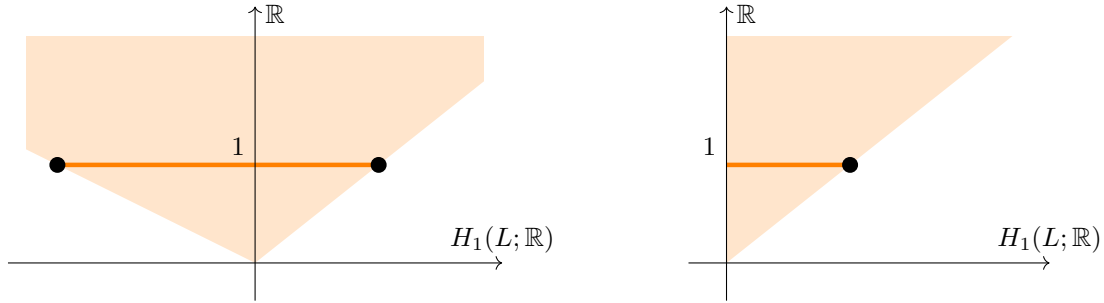
\begin{figure}[ht]
\centering
\begin{tikzpicture}
[blob/.style={circle, draw=black, fill=black, inner sep=0, minimum size=\blobsize}]
\def\blobsize{1.5mm}
\def\l{0.8}
\def\ltot{1.3}
\def\xmax{3}
\def\ymin{-0.1}
\def\ymax{3}
\def\h{1.3}

\begin{scope}
\begin{scope}
\clip (-\xmax, \ymin) rectangle (\xmax, \ymax);
\fill[orange, opacity=0.2] (0, 0) -- ($8*(1, \l)$) -- ++(0, 5) -- ++(-16, 0) -- ($8*(-1, \ltot-\l)$) -- cycle;
\end{scope}
\draw[line width=0.6mm, orange] ($\h/\l*(1, \l)$) node[blob]{} -- ($\h/(\ltot-\l)*(-1, \ltot-\l)$) node[blob]{};
\draw (0, \h) node[anchor=south east]{\small $1$};
\draw[->] (-\xmax-0.25, 0) -- (\xmax+0.25, 0);
\draw[->] (0, \ymin-0.4) -- (0, \ymax+0.4);
\draw (0, \ymax+0.3) node[anchor=west]{\small $\R$};
\draw (\xmax+0.25, 0) node[anchor=south]{\small $H_1(L; \R)$};
\end{scope}

\def\xmin{-0.1}
\def\xmax{4}

\begin{scope}[shift={(6.2, 0)}]
\begin{scope}
\clip (-\xmax, \ymin) rectangle (\xmax, \ymax);
\fill[orange, opacity=0.2] (0, 0) -- ($8*(1, \l)$) -- ++(0, 5) -- ++(-8, 0) -- cycle;
\end{scope}
\draw[line width=0.6mm, orange] ($\h/\l*(1, \l)$) node[blob]{} -- (0, \h);
\draw (0, \h) node[anchor=south east]{\small $1$};
\draw[->] (\xmin-0.4, 0) -- (\xmax+0.25, 0);
\draw[->] (0, \ymin-0.4) -- (0, \ymax+0.4);
\draw (0, \ymax+0.3) node[anchor=west]{\small $\R$};
\draw (\xmax+0.25, 0) node[anchor=south]{\small $H_1(L; \R)$};
\end{scope}

\end{tikzpicture}
\caption{The cone $C$ (shaded), polytope $\Delta^*$ (thick line), and partial boundary $\pd \Delta^*$ (dark blobs) for $\C\P^1$ (left) and $\C$ (right).\label{figDeltaDual}}
\end{figure}

The space $|\Sigma|$ can be described as the subspace of $\R_{\geq 0}^N$ comprising those tuples $(t_1, \dots, t_N)$ with $t_1+\dots+t_N = 1$ such that
\begin{equation}
\label{tjIntersection}
\bigcap_{t_j > 0} F_j \neq \emptyset.
\end{equation}
There is a natural continuous map $\rho$ from this space to $\Delta^*$ given by
\[
\rho(t_1, \dots, t_N) = \frac{t_1}{\lambda_1}(\lambda_1, \nu_1) + \dots + \frac{t_N}{\lambda_N}(\lambda_N, \nu_N)
\]
and we claim that this is a bijection onto $\pd \Delta^*$.  Since the spaces involved are compact Hausdorff, this is enough to show that $\psi$ is a homeomorphism from $|\Sigma|$ to $\pd \Delta^*$ and complete the proof.

Well, viewing $\Delta^*$ as the slice $C \cap (\{1\} \times H_1(L; \R))$, the space $\pd \Delta^*$ consists of those elements of $\Delta^*$ of height zero.  By the height--intersection lemma (\cref{HILemma}), such elements are precisely those sums \eqref{DeltaDual} with $s+t_1+\dots+t_N = 1$ such that $s=0$ and such that \eqref{tjIntersection} is satisfied.  This tells us that $\rho$ lands in $\pd \Delta^*$ and that it is surjective onto this space.  Moreover, height--intersection also tells us that expressions in this sum form are unique, proving that $\rho$ is injective.  Therefore $\rho$ gives a homeomorphism from $|\Sigma|$ to $\pd \Delta^*$ as claimed, proving \cref{SRisCM}.
\end{proof}

\begin{rmk}
Munkres's paper actually deals with the special cases of the sphere and the ball separately \cite[Theorems 2.1 and 2.2]{Munkres}, and for the sphere shows the stronger result that the Stanley--Reisner ring is Gorenstein.
\end{rmk}

\begin{proof}[Proof of \cref{RegSeq}]
Fix an arbitrary basis $\eps_1, \dots, \eps_n$ of $H^1(L; \F)$.  We need to show that the elements
\[
c_j = \sum_{k=1}^N \ip{\eps_j}{\nu_k} Z_k,
\]
form a regular sequence in $SR(X)$.

First note that $SR(X)$ can be graded by giving each $Z_j$ degree $2$ (we could have said degree $1$, but $2$ is more natural when considering the map $\psi : SR(X) \rightarrow H^*(X; \F)$), and that each $c_j$ is homogeneous of degree $2$.  We claim that it suffices to show that the $c_j$ form a regular sequence after localising at the irrelevant ideal $I$ spanned by homogeneous elements of positive degree.

Indeed, suppose that the $c_j$ form a regular sequence in the localisation $SR(X)_I$ and that $a_kc_k$ is zero in $SR(X) / (c_1, \dots, c_{k-1})$ for some $k$ and some $a_k \in SR(X)$.  We are required to show that $a_k$ lies in the ideal $(c_1, \dots, c_{k-1})$ in $SR(X)$, and by our hypothesis we know that it lies in the corresponding ideal in $SR(X)_I$.  Note that each homogeneous part of $a_kc_k$ is zero in $SR(X) / (c_1, \dots, c_{k-1})$, so (by applying the argument that follows to each homogeneous part of $a_k$ in turn) we may assume that $a_k$ itself is homogeneous.

Since $a_k$ lies in the localised ideal $(a_1, \dots, a_{k-1})SR(X)_I$, we can write
\[
a_k = \frac{a_1c_1}{q_1} + \dots + \frac{a_{k-1}c_{k-1}}{q_{k-1}}
\]
for some $a_1, \dots, a_{k-1}$ and $q_1, \dots, q_{k-1}$ in $SR(X)$, where the $q_j$ have non-zero constant term (when expressed as polynomials in the $Z_j$).  Clearing the denominators, singling out the homogeneous part of degree $\deg a_k$, and dividing through by the product of the constant terms in the $q_j$, we obtain an expression for a $a_k$ as an $SR(X)$-linear combination of $a_1, \dots, a_{k-1}$, as desired.  This proves the claim.

We are left to check that $c_1, \dots, c_n$ form a regular sequence in the Noetherian local ring $SR(X)_I$.  By \cref{SRisCM}, this ring is Cohen--Macaulay, so by \cite[Exercise 26.2.D]{VakilAG} it is enough to show that
\[
\dim_\mathrm{K} SR(X)_I/(c_1, \dots, c_n) \leq \dim_\mathrm{K} SR(X)_I - n.
\]
Here the subscript $\mathrm{K}$ indicates Krull dimension.

From the Stanley--Reisner presentation \cref{propSRPresentation} we know that the quotient
\[
SR(X)_I/(c_1, \dots, c_n)
\]
is the localisation of $H^*(X; \F)$ at the irrelevant ideal.  This is just $H^*(X; \F)$ itself, so is finite-dimensional as an $\F$-vector space and hence of Krull dimension zero.  We're now just left to check that $SR(X)_I$ has Krull dimension at least $n$, and to prove this reorder the facets so that $F_1, \dots, F_n$ intersect at some vertex $v$.  Then $SR(X)_I / (Z_{n+1}, \dots, Z_{N})$ is the localisation of $\F[Z_1, \dots, Z_n]$ at the irrelevant ideal, which is $n$-dimensional, so $SR(X)_I$ itself has dimension at least $n$.
\end{proof}

\subsection{The proof of \cref{Theorem1}}
\label{sscProof}

Suppose that we are in the setting of \cref{Theorem1}, namely that we have elements $\hv_1, \dots, \hv_N$ in $\LHD$ satisfying $\hv_j = \lv_j \mod \Lambda_+$ for each $j$.  The rings $\Lambda_0$ and $\LHD$ are not very well-behaved algebraically since the powers of $T$ may be arbitrarily small or large, so in order to make the main algebraic step go through we will `finitise' them, by restricting the allowed powers of $T$ to a submonoid of $\R_{\geq 0}$ and cutting off all powers above a fixed bound.  After introducing these better-behaved rings, $R$ and $S$, we use \cref{RegSeq} to prove the vanishing of a $\Tor$ group (\cref{lemTor}), deduce that the correspondingly finitised Jacobian ring is free as an $R$-module (\cref{lemFreeR}), and then `un-finitise' to obtain \cref{Theorem1}.

First, let $G_0$ denote the submonoid of $\R_{\geq 0}$ generated by the elements $\theta_v(\lambda_j, \nu_j)$ as $j$ ranges over $1, \dots, N$ and $v$ ranges over the vertices of $\Delta$ (if $v$ lies in the facet $F_j$ then $\theta_v(\lambda_j, \nu_j)$ is just $0$).  If $c = (\lambda, \nu)$ is an element of $\Gamma_\R$, then for all vertices $v$ the number $\theta_v(c)$ lies in $h(c) + G_0$, where $h(c)$ is the height of $c$.  To see this, recall from \cref{GammaRSpan} that $c$ can be written as $h(c)(1, 0) + c'$, where $c'$ is a $\Z_{\geq0}$-linear combination of the $(\lambda_j, \nu_j)$, so for each $v$ we have
\[
\theta_v(c) = h(c) + \theta_v(c') \in h(c) + G_0.
\]

\begin{defn}
\label{defnGGamma}
Define $G$ to be the submonoid of $\R_{\geq 0}$ generated by $G_0$ and the heights of the monomials appearing in $\hv_1, \dots, \hv_N$.  Note that $G$ is discrete, since within each $\hv_j$ the heights of the monomials tend to infinity.  Define $\Gamma$ to be the subset of $\Gamma_\R$ comprising elements whose heights lie in $G$.
\end{defn}

\begin{lem}
$\Gamma$ is a submonoid of $\Gamma_\R$, i.e.~it is closed under addition.
\end{lem}
\begin{proof}
Let $c_1$ and $c_2$ be elements of $\Gamma$.  For each vertex $v$ of $\Delta$ the number $\theta_v(c_1)$ lies in $h(c_1) + G_0 \subset G$ by the comment preceding \cref{defnGGamma}, and similarly $\theta_v(c_2)$ lies in $G$.  Writing $h(c_1+c_2)$ as $\min_v \theta_v(c_1)+\theta_v(c_2)$, we deduce that the former lies in $G$.
\end{proof}

The monoids $G$ and $\Gamma$ inherit filtrations from $\R_{\geq 0}$ and $\Gamma_\R$, namely $F^\lambda G = G \cap [\lambda, \infty)$ and $F^\lambda \Gamma = \Gamma \cap (C + (\lambda, 0))$.  These filtrations are inherited in turn by their monoid rings.

\begin{defn}
Fix a positive real number $g$ in $G$.  Define $R$ to be the quotient ring $\F[G] / F^g \F[G]$, and $\m$ to be its maximal ideal $F^{>0}\F[G] / F^g\F[G]$.  Similarly, define $S$ to be $\F[\Gamma] / F^g\F[\Gamma]$.  By construction of $\Gamma$ the ring $S$ contains the elements $\hv_1, \dots, \hv_N$, and we define $J \subset S$ to be the ideal generated by $\sum_j \nu_j \hv_j$.
\end{defn}

Note that $S/F^{>0}S$ is naturally identified with $\LHD/(\Lambda_+ \cdot \LHD) \cong SR(X)$, and that $F^{>0}S$ coincides with $\m S$.  The homological-algebraic statement we will need later is:

\begin{lem}
\label{lemTor}
The group $\Tor_1^R(S/J, R/\m)$ is zero.
\end{lem}
\begin{proof}
We borrow an argument from \cite[Exercise A3.17b]{Eis}.  By \cref{RegSeq} (which was actually inspired by the needs of this argument) and \cref{SRRing}, the components of $\sum_j \nu_j \hv_j$ in $S$ form a regular sequence in $SR(X) = S/\m S$.  Denoting these components by $\widehat{c}_1, \dots, \widehat{c}_n$ (their reductions modulo $\m$ coincide with the $c_1, \dots, c_n$ from the proof of \cref{RegSeq}), we claim that
\[
\Tor_1^R(S/J_k, R/\m) = 0
\]
for all $k$, where $J_k$ is the ideal $(\widehat{c}_1, \dots, \widehat{c}_k) \subset S$.  Taking $k=n$ gives the desired vanishing.

We argue by induction on $k$.  The $k=0$ case follows from the fact that $S$ is free as an $R$-module (a basis is given by the monomials of height $0$).  For the inductive step take $k \geq 1$, assume the result holds for $S/J_{k-1}$, and consider the long exact sequence in $\Tor_*^R(-, R/\m)$ associated to the short exact sequence of $R$-modules
\[
0 \rightarrow \widehat{c}_k(S/J_{k-1}) \rightarrow S/J_{k-1} \rightarrow S/J_k \rightarrow 0.
\]
Combining the inductive hypothesis with this long exact sequence gives
\[
\Tor_1^R(S/J_k, R/\m) \cong \ker \Big( \widehat{c}_k(S/J_{k-1}) \otimes_R R/\m \rightarrow S/J_{k-1} \otimes_R R/\m \Big),
\]
and if we abbreviate $S/J_{k-1}$ to $S'$ then we can express the right-hand side as
\[
\ker ( \widehat{c}_kS' / \m\widehat{c}_kS' \rightarrow S'/\m S') = (\widehat{c}_k S' \cap \m S')/\m\widehat{c}_kS'.
\]
The fact that the $\widehat{c}_j$ form a regular sequence modulo $\m$ means that $\widehat{c}_k$ is not a zero-divisor in $S'/\m S'$, so
\[
\widehat{c}_k S' \cap \m S' = \m\widehat{c}_kS',
\]
completing the inductive step and in turn proving the lemma.
\end{proof}

Now fix height zero monomials $e_1, \dots, e_m$ in $\LHD$ which form an $\F$-basis for
\[
\LHD \Big/ \Big( \Lambda_+ \cdot \LHD + \Big( \sum_{j=1}^N \nu_j \hv_j \Big) \Big) \cong S/(\m S + J) \cong H^*(X; \F).
\]
Note that we can also view the $e_j$ as elements of $S$, which we similarly denote by $e_j$.  Our ultimate claim is that the $e_j$ form a free $\Lambda_0$-basis for $\LHD / (\sum_j \nu_j \hv_j)$---this is the content of \cref{Theorem1}---but first we prove the corresponding statement for $S$:

\begin{prop}
\label{lemFreeR}
The $e_j$ form a free $R$-basis for $S / J$.
\end{prop}
\begin{proof}
The ring $S$ is spanned over $\F$ by monomials $s$ corresponding to elements of $\Gamma$ of height less than $g$.  In order to prove that the $e_j$ span $S/J$ over $R$ it therefore suffices, by downward induction on height, to show that any such $s$ can be written as an $R$-linear combination of the $e_j$ plus an error term in $J$ and an error term of height strictly greater than $h(s)$.

So fix an arbitrary $s$, and write it in the form $T^{h(s)} s_0$, where $s_0$ is a monomial in $S$ of height zero.  Since the $e_j$ span $S/(\m S+J)$, there exist $a_1, \dots, a_m$ in $\F$ such that
\[
s_0 - (a_1 e_1 + \dots + a_m e_m)
\]
lies in $\m S + J$.  This means that $s$ is given by $(T^{h(s)}a_1) e_1 + \dots + (T^{h(s)}a_m) e_m$ plus error terms in $J$ and in $T^{h(s)}\m S$.  The latter clearly has height greater than $h(s)$, so our inductive argument is complete and the $e_j$ do indeed span $S/J$ over $R$.

Now consider the natural map $i : F \rightarrow S/J$, where $F$ is the free $R$-module on the basis $e_j$.  We have just seen that $i$ is surjective, so letting $K$ denote its kernel we have a short exact sequence
\[
0 \rightarrow K \rightarrow F \rightarrow S/J \rightarrow 0.
\]
Reducing modulo $\m$, the map $F \rightarrow S/J$ becomes an isomorphism (because the $e_j$ form an $\F$-basis for $S/(\m S+J)$), so we obtain an exact sequence
\[
\dots \rightarrow \Tor_1^R(S/J, R/\m) \rightarrow K/\m K \rightarrow 0.
\]
The $\Tor$ group vanishes by \cref{lemTor}, and so $K = \m K$, which then forces $K$ to vanish---for example, consider a putative non-zero element of minimal height to obtain a contradiction; or use the fact that $K = \m^j K$ for all $j$ and that $\m^j = 0$ for $j$ sufficiently large; or simply appeal to Nakayama's lemma.  Therefore $i$ is an isomorphism of $R$-modules, as claimed.
\end{proof}

And finally we reach the main goal:

\begin{proof}[Proof of \cref{Theorem1}]
Both $R$ and $\Lambda_0$ are algebras over $\F[G]$, and we have natural identifications
\[
R \otimes_{\F[G]} \Lambda_0 \cong \Lambda_0 / T^g \Lambda_0 \text{\quad and \quad} S \otimes_{\F[G]} \Lambda_0 = \LHD / T^g \LHD.
\]
Tensoring \cref{lemFreeR} with $\Lambda_0$ over $\F[G]$ therefore tells us that
\begin{equation}
\label{modTg}
\LHD \Big/ \Big( T^g \LHD + \Big( \sum_{j=1}^N \nu_j \hv_j \Big) \Big)
\end{equation}
is a free $\Lambda_0 / T^g \Lambda_0$-module on the basis $e_1, \dots, e_m$.  Since $g$ was an arbitrary element of $G$, we can take the (inverse) limit as $g \rightarrow \infty$, and see that the limit of \eqref{modTg} is a free $\Lambda_0$-module on $e_1, \dots, e_m$.  This limit carries a natural map from $\LHD$, which is surjective since the latter is $T$-adically complete, and the kernel is precisely
\[
\bigcap_{g \in G} \Big( T^g \LHD + \Big( \sum_{j=1}^N \nu_j \hv_j \Big) \Big) = \clos \Big( \sum_{j=1}^N \nu_j \hv_j \Big).
\]
We conclude that
\[
\LHD \Big/ \clos \Big( \sum_{j=1}^N \nu_j \hv_j \Big)
\]
is free over $\Lambda_0$.
\end{proof}

\section{Further discussion}
\label{secFurtherDiscussion}

\subsection{Non-compactness}
\label{sscNonCompact}

In compact symplectic manifolds $X$, the standard Gromov compactness results \cite[Section 1.5.B]{Gro} ensure that the sums defining the quantum cohomology product converge, by guaranteeing that moduli spaces of pseudoholomorphic spheres of bounded area can be compactified with bubbled curves.  When $X$ is non-compact this no longer need be the case, since sequences of curves can escape (consider $\C\P^2$ minus a point, and a family of lines moving towards that point, for a trivial example).  One therefore needs to place conditions on $X$, and the chosen almost complex structure $J$, at its ends in order to prevent this.

Explicitly, the property we need is as follows: for any compact set $K \subset X$ and any positive real number $E$, there exists a compact set $K' \subset X$ such that any stable $J$-holomorphic sphere in $X$ which meets $K$ and has area at most $E$ is contained in $K'$.  Let $\mathcal{J}_{QH}$ denote the set of $\omega$-compatible almost complex structures on $X$ that satisfy this condition (whenever we talk about almost complex structures they will implicitly be assumed to be $\omega$-compatible).  The Kuranishi machinery of \cite{FukayaOno} allows one to define quantum cohomology with respect to any $J$ in $\mathcal{J}_{QH}$ (which may be empty), and by considering cobordisms between moduli spaces of curves this definition depends only on the path-component of $J$ in $\mathcal{J}_{QH}$.

Quantum cohomology defined in this way is not then a purely symplectic invariant in general, but there are various extra assumptions and choices one can make to ensure that $\mathcal{J}_{QH}$ is non-empty and to single out a distinguished component.  For example, in the toric case we have the following:

\begin{lem}
\label{lemStdJ}
For toric $X$ the set $\mathcal{J}_{QH}$ contains the standard integrable complex structure.
\end{lem}
\begin{proof}
Recall from \cref{sscToricBG} that the algebro-geometric description of $X$ is as $\Proj \C[\Sigma_\Z]$ where $\Sigma_\Z$ is a submonoid of $A^* \times \Z_{\geq 0}$, graded by its $\Z_{\geq 0}$-component.  It is therefore projective over the affine variety $\Spec \C[\Sigma_\Z]_0$, where the subscript $0$ denotes the degree zero part.  The maximum principle applies in the affine base, while the fibres are compact, so the required property holds.
\end{proof}

The same argument shows that the moduli spaces of holomorphic discs (with respect to the standard complex structure) used to construct the Kodaira--Spencer map and prove \cref{propFOOOks} also have the necessary compactness properties.  The upshot of this is that everything we have done applies to non-compact toric varieties just as it does to compact ones, once we interpret the quantum cohomology as that corresponding to the component of $\mathcal{J}_{QH}$ containing the standard complex structure.  The use of the standard $J$, and in particular its torus-equivariance, is actually crucial in the technical work of Fukaya--Oh--Ohta--Ono which underlies \cref{propFOOOks}, so understanding quantum cohomology for other components of $\mathcal{J}_{QH}$ using similar methods seems out of reach at present (except when $X$ is monotone; see \cref{NonCpctMonotone}).

\subsection{Monotonicity}
\label{sscMonotonicity}

Recall that symplectic manifold $X$ is \emph{monotone} if its first Chern class is a positive multiple of the class $[\omega]$ of the symplectic form in $H^2(X; \R)$, and a Lagrangian submanifold $L \subset X$ is monotone if its Maslov index homomorphism $H_2(X, L) \rightarrow \Z$ is a positive multiple of $[\omega]$ as classes in $H^2(X, L; \R)$.  These notions are respectively the symplectic equivalent of Fano and its relative version.  Restricting to our usual setup where $X$ is toric and $L$ is the toric fibre over $0$ we have the following well-known result:

\begin{lem}
$L$ is monotone if and only if the $\lambda_j$ are all equal, and $X$ is monotone if and only if $\Delta$ can be translated to make this the case (keeping all $\lambda_j$ positive), i.e.~if and only if it has a monotone toric fibre.  $X$ has at most one monotone fibre, except in the case $X = \C^n$ ($n \geq 1$ arbitrary) where it has a one-parameter family of monotone fibres.
\end{lem}
\begin{proof}
The basic disc classes $[u_1], \dots, [u_N]$ form a basis for $H_2(X, L; \Z)$, and have areas $\lambda_1, \dots, \lambda_N$.  Moreover, they all have index $2$, so $L$ is monotone if and only if the $\lambda_j$ are all equal.

The toric divisor classes $H_1, \dots, H_N$ in $H^2(X, L; \Z)$ are dual to the $[u_j]$ so we have
\[
[\omega] = \sum_{j=1}^N \lambda_j H_j
\]
in $H^2(X, L; \R)$.  We also know that the first Chern class of $X$ is simply the sum of the $H_j$.  Therefore $X$ is monotone if and only if there exists a positive real number $\lambda$ such that
\[
\lambda \sum_{j=1}^N H_j = \sum_{j=1}^N \lambda_j H_j
\]
in $H^2(X; \R)$.  From the long exact sequence of the pair $(X, L)$, this holds if and only if there exists $b$ in $H^1(L; \R)$ such that
\[
\sum_{j=1}^N (\lambda_j-\lambda) H_j = b \circ \pd
\]
on $H_2(X, L; \R)$, i.e.~such that for all $j$ we have $\lambda_j = \lambda + \ip{b}{\nu_j}$.  Replacing each $\lambda_j$ by $\lambda_j - \ip{b}{\nu_j}$ corresponds exactly to translating $\Delta$ by $b$, so we see that $X$ is monotone if and only if we can make the $\lambda_j$ all equal and positive by translating $\Delta$.

Finally we deal with uniqueness.  For $X = \C^n$, with polyhedron $\{x \in \R^n : x_j \geq -\lambda_j\}$, we know that the fibre over $0$ is monotone if and only if the $\lambda_j$ are equal and positive, so the monotone fibres are parametrised by this common value in $\R_{>0}$.  Geometrically, the monotone fibres lie over a ray emanating from the vertex of the polyhedron.  For general $X$, there is an analogous ray at each vertex of its polyhedron, and any monotone fibre must lie on all of these rays.  When there are at least two vertices (i.e.~when $X$ is not of the form $\C^n$) the intersection of the rays consists of at most one point, so there is at most one monotone fibre.
\end{proof}

For the rest of this subsection we will assume that $X$ is monotone and that $L$ is the (almost) unique monotone toric fibre.  By rescaling $\omega$ we will assume for simplicity that the $\lambda_j$ are all $1$.  We temporarily also restrict to compact $X$---the non-compact case is considered in \cref{sscMonNonCpct}.

Floer theory is much simpler in this setting, and virtual perturbation theory is not necessary.  Everything we need is contained in the foundational paper \cite{BCQS} of Biran and Cornea, where transversality is achieved by varying the almost complex structure and imposing a small Hamiltonian perturbation at the interior input defining $\CO^0$.  In particular, the quantum cohomology of $X$ can be defined over $R[T]$, for any ground ring $R$, whilst the self-Floer cohomology of $L$ (referred to as \emph{Lagrangian quantum homology} in \cite{BCQS}) can be defined over $R[\Gamma]$, where $\Gamma$ is the submonoid of $\Z \oplus H_1(L; \Z)$ generated by $(1, \nu_1), \dots, (1, \nu_n)$ and $(1, 0)$.  No completions are needed.  Note, however, that this Floer group $HF^*(L, L; R[\Gamma])$ is not invariant under Hamiltonian isotopies of one of the copies of $L$ (the coefficient system doesn't really make sense unless the two Lagrangians are equal), and is always non-zero, even if $L$ is displaceable.

\begin{rmk}
Non-vanishing of the group $HF^*(L, L; R[\Gamma]) \otimes_{R[\Gamma]} R[\Z \oplus H_1(L; \Z)]$ obtained by adjoining inverses to the monomials in $R[\Gamma]$ \emph{does} prove non-displaceability of $L$ though.  This follows from the framework of Zapolsky \cite{ZapPSS}, in which the coefficient ring $R[\Z \oplus H_1(L; \Z)]$ can be interpreted as a local system and then compared with the Lagrangian intersection picture of $HF^*$ using his PSS morphism.
\end{rmk}

We can compute $HF^*(L, L; R[\Gamma])$, as sketched in \cref{lemHF} in the general case, but now using the Oh spectral sequence
\[
E_1 = H^*(L; R[\Gamma]) \implies HF^*(L, L; R[\Gamma]),
\]
which is induced by our familiar height filtration on $R[\Gamma]$.  Since $H^*(L; R[\Gamma])$ is generated as a $R[\Gamma]$-algebra by $H^1(L; R)$, the subalgebra of $HF^*(L, L; R[\Gamma])$ generated by the unit $1_L$ is the quotient of $R[\Gamma]$ by the ideal generated by the images of the index $2$ differential $H^1(L; R) \rightarrow H^0(L; R[\Gamma])$ on $E_1$.  This differential sends a class $b$ to $\sum_j \ip{b}{\nu_j} \lv_j$, where $\lv_j = T \tau^{\nu_j}$ as usual (the $T$ here is really $T^{\lambda_j}$, but recall that $\lambda_j = 1$), so the required ideal is $(\sum_j \nu_j \lv_j)$.

The map
\[
\CO^0 : QH^*(X; R[T]) \rightarrow HF^*(L, L; R[\Gamma])
\]
(called the \emph{quantum module action} in \cite{BCQS}) is completely rigorous in this monotone setting, and one computes straightforwardly that $\CO^0(H_j) = \lv_j \cdot 1_L$---for degree reasons the only contributions come from index $2$ discs, and the count can be done using the standard complex structure for which the only such discs are basic ones $[u_j]$ (see \cite[Proposition 4.6.7]{SmDCS} for a more general result, including verification of the signs).  Since the $H_j$ generate $QH^*(X; R[T])$ as a $R[T]$-algebra, the whole image of $\CO^0$ is contained in $(R[\Gamma] / (\sum_j \nu_j \lv_j)) \cdot 1_L$, without appealing to any torus-equivariance as we had to do in the general case, and we obtain a \emph{monotone Kodaira--Spencer map}
\[
\ks_\mathrm{mon} : QH^*(X; R[T]) \rightarrow R[\Gamma] \Big/ \Big(\sum_j \nu_j \lv_j\Big).
\]
This is a homomorphism of $R[T]$-algebras sending $H_j$ to $\lv_j$, giving the monotone analogue of \cref{propFOOOks}.

\begin{rmk}
To get to here, the only pseudoholomorphic curves we have had to consider explicitly are index $2$ discs bounded by $L$, and, as mentioned, for these we can work with the standard complex structure.  There is no need to worry about curves with components contained in the toric divisor, and the potential failure of transversality for such curves is what makes a direct consideration of holomorphic spheres in $X$, and hence a direct computation of $QH^*(X)$, difficult.
\end{rmk}

The main result is then:

\begin{prop}
\label{MonotoneQH}
For any compact monotone toric variety $X$, the map $\ks_\mathrm{mon}$ is an isomorphism of $R[T]$-algebras.
\end{prop}
\begin{proof}
Assume first that $R$ is an integral domain, and let $\F$ denote its fraction field.  Surjectivity is automatic, since $\ks_\mathrm{mon}$ hits the generators $\lv_j$ of the codomain.  To prove injectivity, note that since $QH^*(X; R[T])$ is a free $R$-module (in fact, a free $R[T]$-module) it injects into its tensor product with $\F$, so it suffices to show that $\F \otimes_\R \ker \ks_\mathrm{mon}$ vanishes.  By exactness of $\F \otimes_R -$, it is therefore enough to show that the induced map
\[
\ks_\mathrm{mon}^\F : QH^*(X; \F[T]) \rightarrow \F[\Gamma] \Big/ \Big(\sum_j \nu_j \lv_j\Big)
\]
is injective.

Now that we are working over a field, we can argue essentially as in \cref{secProof}, although things are simpler because there is no need to finitise the rings.  Explicitly, we can apply \cref{RegSeq} to see that the generators of $(\sum_j \nu_j \lv_j)$ form a regular sequence modulo $T$ in $\F[\Gamma]$, and then use the argument of \cref{lemTor} to obtain the vanishing of the group
\[
\Tor_1^{\F[T]}\bigg(\F[\Gamma] \Big/ \Big(\sum_j \nu_j \lv_j\Big), \F[T]/(T)\bigg).
\]
Applying this to the long exact sequence in $\Tor_*^{\F[T]}(-, \F[T]/(T))$ obtained from
\[
0 \rightarrow \ker \ks_\mathrm{mon}^\F \rightarrow QH^*(X; \F[T]) \xrightarrow{\ \ks_\mathrm{mon}^\F \ } \F[\Gamma] \Big/ \Big(\sum_j \nu_j \lv_j\Big) \rightarrow 0,
\]
and using the Stanley--Reisner presentation to see that $\ks_\mathrm{mon}^\F \otimes \F[T]/(T)$ is an isomorphism, we get $\ker \ks_\mathrm{mon}^\F = T \ker \ks_\mathrm{mon}^\F$.  This then yields $\ker \ks_\mathrm{mon}^\F = 0$, which is exactly what we want.

To extend from integral domains to arbitrary rings $R$, take the result over $\Z$ and tensor with $R$.  (For our present purposes then, we may as well have restricted to the case $R=\Z$, $\F = \Q$ in the above argument.  The case of arbitrary integral domains will be used later though, in \cref{sscBFields}.)
\end{proof}

\subsection{Non-compactness in the monotone case}
\label{sscMonNonCpct}

We have just seen how our results can be refined when $X$ is monotone, under the assumption that it is compact, so we now discuss how to extend to the non-compact case.  Recall from \cref{sscNonCompact} that in order to define quantum cohomology of non-compact manifolds we need to restrict to almost complex structures $J$ in the subset $\mathcal{J}_{QH}$, to prevent $J$-holomorphic curves from escaping to infinity, and that a priori the resulting cohomology ring depends on the path-component of $J$ in $\mathcal{J}_{QH}$.  Similarly, to count the discs needed to define the Floer cohomology of $L$ and the closed--open string map we need to restrict to the subset $\mathcal{J}_{HF}$ of $\mathcal{J}_{QH}$, comprising those $J$ with the following additional property: for any positive real number $E$ there exists a compact set $K \subset X$ such that any stable $J$-holomorphic disc in $X$ with boundary on $L$ and of area at most $E$ is contained in $K$.

Unfortunately, such $J$ may not be compatible with the machinery of Biran--Cornea used to construct $\ks_\mathrm{mon}$ above, since this requires the almost complex structure to be suitably generic.  There are two ways around this: impose geometric conditions to ensure that some such $J$ \emph{are} suitable; or appeal to sophisticated virtual perturbation theory to allow any $J$ in $\mathcal{J}_{HF}$ to be used.  We discuss each of these in turn.

In the first direction, the standard approach to Floer theory on a (not necessarily toric) non-compact symplectic manifold $X$ is to assume that it is \emph{conical at infinity}.  This means that it is equipped with a decomposition, up to an appropriate notion of equivalence, of the form
\[
X = X_\mathrm{in} \cup_{(Y, \alpha)} X_\mathrm{con}.
\]
Here $X_\mathrm{in}$ is a compact codimension-$0$ submanifold of $X$ (the `interior') with boundary $Y$, $\alpha$ is a contact form on $Y$, and $X_\mathrm{con}$ (the `conical end') is given by the positive symplectisation
\[
\lb Y \times [0, \infty) , \diff (e^r \alpha) \rb
\]
of $Y$, where $r$ is the radial coordinate in $[0, \infty)$.  The conical end carries the Liouville vector field $Z$ given by $\pd_r$, and an almost complex structure $J$ on $X$ is \emph{of contact type} if on the conical end we have
\[
\diff(e^r) \circ J = - e^r \alpha
\]
for sufficiently large $r$.  We denote the set of such $J$ by $\mathcal{J}_\mathrm{con}$.  Crucially, a maximum principle applies to $J$-holomorphic curves for all contact-type $J$, which implies that $\mathcal{J}_\mathrm{con}$ is contained in $\mathcal{J}_{QH}$.  Moreover, any two such $J$ can be connected by a path of contact-type almost complex structures, so there is a unique contact-type component of $\mathcal{J}_{QH}$.  A classic reference for these concepts is Seidel's survey of symplectic cohomology \cite{SeidelBiasedView}, for which this is the natural setting (it is also the setting for the open-string version---wrapped Floer theory).

Specialising to toric $X$, and assuming that our monotone toric fibre $L$ is contained in the interior of $X_\mathrm{in}$ (which we can always do by moving its boundary $Y$ outwards along the conical end), we have the following standard result:

\begin{lem}
\label{lemconHF}
The set $\mathcal{J}_\mathrm{con}$ is contained in $\mathcal{J}_{HF}$.  Moreover, any $J$ in $\mathcal{J}_\mathrm{con}$, and any path between two such $J$, can be perturbed within this class to achieve the transversality needed by Biran--Cornea.
\end{lem}
\begin{proof}[Sketch proof]
Take $J$ in $\mathcal{J}_\mathrm{con}$.  The fact that $J$ is contained in $\mathcal{J}_{HF}$ follows immediately from the maximum principle.  This principle also ensures that non-constant $J$-holomorphic spheres must enter $X_\mathrm{in} \subset X$, and clearly the same is true for $J$-holomorphic discs with boundary on $L$.  Transversality for these curves can therefore be achieved by small perturbations of $J$ on $X_\mathrm{in}$, which do not spoil its contact type at infinity.  Paths in $\mathcal{J}_\mathrm{con}$ can be perturbed similarly, completing the proof.  There is a small complication, in that Biran--Cornea actually need to introduce Hamiltonian perturbations on discs with interior marked points (to deal with the possibility of such discs being non-simple---see \cite[Section 5.3.6]{BCQS} and the paragraphs immediately preceding it for what goes wrong and how Hamiltonian perturbations can fix it), but these perturbations can be chosen in a way that preserves the maximum principle \cite[Section (8a)]{SeidelBiasedView}.
\end{proof}

The result of this is:

\begin{prop}
For any monotone toric variety $X$ (whose moment polyhedron has a vertex) which is conical at infinity, using only classical transversality techniques there is a well-defined `contact type at infinity' interpretation of $QH^*(X; R[T])$, and it is given by \cref{MonotoneQH}.
\end{prop}
\begin{proof}
Everything goes through as in the compact case, now using the maximum principle to prevent curves from escaping.  The only point that needs some care is to check that the superpotential defined using contact-type $J_\mathrm{con}$  (for which all index $2$ discs are regular) agrees with that defined using the standard $J_\mathrm{std}$, since the former appears in the construction of $\ks_\mathrm{mon}$ but the latter appears in the main algebraic result that the Jacobian ring is a free module.  We now explain why this is indeed the case.  The author is indebted to Alex Ritter for making the crucial suggestion of switching the roles of $J_\mathrm{std}$ and $J_\mathrm{con}$ in the following argument.

First note that it suffices to consider a single $J_\mathrm{con}$, since the standard cobordism argument shows that the superpotential is invariant under homotopies between different choices.  Moreover, we may make this choice so that $J_\mathrm{con}$ is \emph{conical}, meaning that it is invariant under the Liouville vector field $\pd_r$ on $X_\mathrm{con}$ for sufficiently large $r$.  This is because a choice of conical contact-type almost complex structure always exists (choose it along a slice $r = \text{const}$, extend it outwards using the Liouville flow, and extend it inwards arbitrarily), and the same argument used in \cref{lemconHF} shows that we can perturb it to obtain regularity without spoiling these properties.

Under this assumption, for any $E > 0$ there exists an $R > 0$ such that any $J_\mathrm{con}$-holomorphic curve with boundary on both $r = R$ and $r = 2R$ has area at least $E$---see \cite[Lemma 1]{OanceaKunneth}.  In particular, any $J_\mathrm{con}$-holomorphic disc on $L$ of area less than $E$ cannot cross $r = 2R$.  Choosing $E$ greater than the area of all index $2$ disc classes (which we can do since $X$ is monotone) we conclude that for any almost complex structure $J'$ on $X$ which coincides with $J_\mathrm{con}$ on the compact set $K$ bounded by $r = 2R$, all $J'$-holomorphic discs of index $2$ are contained within $K$.

Recall from \cref{lemStdJ} that $(X, J_\mathrm{std})$ carries a holomorphic map to some affine space $\C^m$, with projective fibres.  Let $\pi$ denote this map, and take a compact set $C$ in $\C^m$ which contains a neighbourhood of $\pi(K)$.  Now choose $J'$ as in the preceding paragraph which coincides with $J_\mathrm{std}$ outside $\pi^{-1}(C)$.  Pick a generic path $J_t$ from $J'$ to $J_\mathrm{std}$ which is constant outside $\pi^{-1}(C)$, and consider the resulting cobordism of moduli spaces of index $2$ discs.  Gromov compactness applies to this cobordism since any $J_t$-holomorphic disc must stay within $\pi^{-1}(C)$, by considering its projection under $\pi$ and applying the maximum principle in $\C^m$.  We deduce that the $J'$- and $J_\mathrm{std}$-superpotentials agree, and hence that the $J_\mathrm{con}$- and $J_\mathrm{std}$-superpotentials also agree.
\end{proof}

The more powerful (but technology-heavy) approach is to use the results of Fukaya--Oh--Ohta--Ono in \cite{FOOOIntegers}, which show that for (spin) monotone Lagrangians the moduli spaces of discs used in Floer theory can be equipped with virtual fundamental cycles over $\Z$, not just $\Q$.  Letting $\mathcal{J}_W$ denote the union of those components of $\mathcal{J}_{HF}$ on which the superpotential agrees with the standard one, we obtain:

\begin{prop}
\label{NonCpctMonotone}
For any monotone toric variety $X$ (whose moment polyhedron has a vertex), using virtual perturbation techniques $QH^*(X; R[T])$ can be defined for any component of $\mathcal{J}_{QH}$, and for any component meeting $\mathcal{J}_W$ it is given by \cref{MonotoneQH}.  In particular, this holds for the component containing the standard integrable $J$, and for the component corresponding to contact-type $J$ if $X$ is conical at infinity.
\end{prop}

This has the non-obvious consequence that, under the hypotheses of \cref{MonotoneQH}, the ring $QH^*(X; R[T])$ is independent of the choice of component of $\mathcal{J}_{QH}$ used to define it, amongst those components which meet $\mathcal{J}_W$.

As a final comment, the reader may worry that the construction of \emph{integral} virtual fundamental cycles in \cite{FOOOIntegers} may not be compatible with the construction of \emph{torus-equivariant} virtual fundamental cycles in \cite{FOOOMSToric}.  In other words, integrality may be inconsistent with equivariance.  However, as remarked earlier in the discussion of the compact case, torus-equivariance is not needed for monotone $X$: the results which which rely on equivariance for general $X$ hold for simple degree reasons in the monotone setting.

\subsection{Bulk deformations}
\label{sscBulk}

Bulk deformations were introduced by Fukaya--Oh--Ohta--Ono in \cite{FOOObig} to provide additional flexibility in Floer theory.  The idea is as follows: given a class $\mathfrak{b}$ in $H^\mathrm{even}(X; \Lambda_+)$, one replaces each count of pseudoholomorphic curves, which we'll denote by $\#$, with the sum
\[
\sum_{k=0}^\infty \frac{1}{k!} \#(\mathfrak{b}, k),
\]
where $\#(\mathfrak{b}, k)$ schematically counts the same curves as $\#$ but with $k$ (movable) interior marked points introduced and constrained to lie on cycles Poincar\'e dual to $\mathfrak{b}$.  The fact that $\mathfrak{b}$ carries positive Novikov weight (meaning that its coefficients lie in the positive part $\Lambda_+$ of the Novikov ring) ensures that this sum converges, and that its leading order term is $\#(\mathfrak{b}, 0)$, which is the same as the undeformed count $\#$ (as long as the same perturbations are used).

Via this construction one can define bulk-deformed quantum cohomology $QH^*(X, \mathfrak{b}; \Lambda_0)$, which still coincides with the classical cohomology modulo $\Lambda_+$, and bulk-deformed Lagrangian Floer cohomology.  The family of bulk-deformed quantum  cohomology rings is what is usually referred to as the \emph{big} quantum cohomology.  In our setting of toric $X$, there is also a bulk-deformed superpotential $W^\mathfrak{b} = W_1^\mathfrak{b} + \dots + W_N^\mathfrak{b}$ in $\LHD$ and a bulk-deformed Kodaira--Spencer map $\ks_\mathfrak{b}$ (constructed in \cite{FOOOMSToric}), satisfying the natural analogue of \cref{propFOOOks}:

\begin{prop}
\label{propBulkks}
Each $W_j^\mathfrak{b}$ has
\[
W_j^\mathfrak{b} = \lv_j \mod \Lambda_+,
\]
and $\ks_\mathfrak{b}$ is a $\Lambda_0$-algebra homomorphism
\[
QH^*(X, \mathfrak{b}; \Lambda_0) \rightarrow \LHD \Big/ \Big(\sum_{j=1}^N \nu_j W_j^\mathfrak{b} \Big)
\]
which sends $H_j$ to $W_j^\mathfrak{b}$.
\end{prop}

The verification of these properties follows the same path as outlined in the undeformed case in \cref{secks}, since the equality $\#(\mathfrak{b}, 0) = \#$ means that the deformation does not affect any of the leading-order terms.  The only point where one has to take extra care is in the proof of convergence, in \cref{Convergence} and the paragraph preceding it.  There we said that each time we count rigid curves we can restrict to a fixed index, but this is no longer true if $\mathfrak{b}$ lies outside $H^2(X; \Lambda_+)$, because introducing an interior marked point and constraining it to a cycle of codimension $d$ changes the virtual dimension of the moduli space by $2-d$.  To fix this we should break up each sum according to the number of insertions of $\mathfrak{b}$.  Within each sub-sum the indices of the discs are bounded and we have convergence as before, and since $\mathfrak{b}$ carries positive Novikov weight convergence is preserved when we put all of the sub-sums back together.

By combining \cref{propBulkks} with \cref{Theorem1} we obtain the bulk-deformed version of \cref{corksIso}, asserting that $\ks_\mathfrak{b}$ is an isomorphism onto the generalised Jacobian ring of $W^\mathfrak{b}$ (i.e.~after taking the closure of the ideal defining the codomain of $\ks_\mathfrak{b}$).

\begin{rmk}
\label{rmkBulk}
In \cite{FOOObig} slightly more general bulk classes are allowed, of the form $\mathfrak{b}_2 + \mathfrak{b}_\text{high}$, where $\mathfrak{b}_2$ lies in $H^2(X; \Lambda_0)$ and $\mathfrak{b}_\text{high}$ lies in $H^{\mathrm{even},\geq 4}(X; \Lambda_+)$.  Compared with our treatment, this just means that the $H^2$-component of $\mathfrak{b}$ may have zero Novikov weight.  This extra deformation, arising from a class $B$ in $H^2(X; \C)$, has to be imposed by hand by weighting the count of curves $u$ by $e^{\int u^* B}$---see \cite[Equation (11.4)]{FOOObulk}.  We discuss such deformations separately, under the name $B$-fields, in the following subsection.

These constructions are not specific to the toric case, but for toric fibres Fukaya--Oh--Ohta--Ono \cite{FOOOMSToric} show that one can actually go further and achieve convergence even when $\mathfrak{b}_\text{high}$ has zero Novikov weight.  Our methods do not cover this extension (what goes wrong is the proof of convergence, since the indices of the discs being counted are unbounded and we are not rescued by powers of $T$ as we were above) but its geometric meaning is unclear and to the author's knowledge it is not used in any applications.
\end{rmk}

\subsection{$B$-fields}
\label{sscBFields}

Given a class $\rho$ in $H^2(X; \C^*)$ one can define another deformation of quantum cohomology, in which each curve $u$ being counted is weighted by $\ip{\rho}{[u]}$.  This is equivalent to taking a closed, complex-valued $2$-form $B$ on $X$ and weighting counts by $e^{\int u^*B}$.  Such a $B$ is usually called a $B$-field, and heuristically one can think of the class $[B]$ in $H^2(X; \C)$ as a degree $2$ bulk deformation, but to avoid convergence issues it is more convenient to view it in terms of $\rho$.

Restricting once more to toric $X$, with $L$ a toric fibre, we have that for any abelian group $A$ the restriction map $H^2(X; A) \rightarrow H^2(L; A)$ is zero, since we saw in \cref{sscToricBG} that $H^2(X; A)$ is spanned by the Poincar\'e duals of the toric divisors, which are disjoint from $L$.  We can therefore lift any class $\rho$ in $H^2(X; \C^*)$ (non-uniquely) to a class $\widehat{\rho}$ in $H^2(X, L; \C^*)$ and use this to weight counts of discs with boundary on $L$.  This allows us to define a $\widehat{\rho}$-twisted version of $\ks$.

We obtain deformations $QH^*(X, \rho; \Lambda_0)$ of quantum cohomology, $W^\rho = W_1^\rho + \dots + W_N^\rho$ of the superpotential in $\LHD$, and $\ks_\rho$ of the Kodaira--Spencer map (these depend on the choice of lift $\widehat\rho$, but we don't explicitly notate this).  Letting $\rho_j = \ip{\widehat{\rho}}{[u_j]}$, the deformed version of \cref{propFOOOks} now reads:

\begin{prop}
\label{propBks}
Each $W_j^\rho$ has
\[
W_j^\rho = \rho_j\lv_j \mod \Lambda_+,
\]
and $\ks_\rho$ is a $\Lambda_0$-algebra homomorphism
\[
QH^*(X, \rho; \Lambda_0) \rightarrow \LHD \Big/ \Big(\sum_{j=1}^N \nu_j W_j^\rho \Big)
\]
which sends $H_j$ to $W_j^\rho$.
\end{prop}

In order to prove that $\ks_\rho$ is an isomorphism after closing the ideal in the codomain we need the following generalisation of \cref{Theorem1}:

\begin{prop}
\label{BFieldQH}
For any $\rho_1, \dots, \rho_N$ in $\C^*$ and any elements $\hv_1, \dots, \hv_N$ of $\LHD$ satisfying
\[
\hv_j = \rho_j \lv_j \mod \Lambda_+
\]
for all $j$, the $\Lambda_0$-algebra
\[
\LHD \Big/ \clos\Big(\sum_{j=1}^N \nu_j \hv_j\Big)
\]
is free as a $\Lambda_0$-module.
\end{prop}
\begin{proof}
Recall the fundamental identification  between $SR(X)$ and $\LHD / (\Lambda_+ \cdot \LHD)$ given in \cref{SRRing} by $Z_j \mapsto \lv_j$.  Since the ideal defining $SR(X)$ is generated by monomials in the $Z_j$, the map $Z_j \mapsto \rho_j Z_j$ defines an automorphism, and we deduce that $SR(X)$ can also be identified with $\LHD / (\Lambda_+ \cdot \LHD)$ via $Z_j \mapsto \rho_j \lv_j$.  Using this modified identification, the proof of \cref{Theorem1} goes through to prove the required generalisation.
\end{proof}

In exactly the same way, \cref{corModIso} correspondingly generalises, and we obtain the claimed isomorphism from $\ks_ \rho$.  The `$B$-field' deformation $\rho$ can be combined with bulk deformations to give deformations of the form $\mathfrak{b}_2 + \mathfrak{b}_\text{high}$ appearing in \ref{rmkBulk} (but with $\mathfrak{b}_\text{high}$ still required to have positive Novikov weight), and the previous arguments combine to show that once again the deformed Kodaira--Spencer map induces an isomorphism after closing the ideal in the codomain.

For monotone $X$, without deformations, we can work over an arbitrary ground ring $R$ and replace $\Lambda_0$ with $R[T]$.  Bulk deformations cannot be formulated in this simplified algebraic setting (they require both infinite sums and division by factorials, so one should at least require that $R$ contains $\Q$ and complete $R[T]$ to $R\llbracket T \rrbracket$), but the obvious analogue of the $B$-field deformation by a class $\rho$ in $H^2(X; R^\times)$ \emph{does} make sense.  The proof of \cref{MonotoneQH} can be modified using the ideas of \cref{BFieldQH} to prove:

\begin{prop}
For any monotone toric variety $X$ whose moment polyhedron has a vertex, and for any integral domain $R$ and class $\rho \in H^2(X; R^\times)$, the map
\[
\ks_{\rho, \mathrm{mon}} : QH^*(X, \rho; R[T]) \rightarrow R[\Gamma] \Big/ \Big(\sum_{j=1}^N \nu_j W_j^\rho \Big)
\]
is an isomorphism of $R[T]$-algebras.
\end{prop}

The restriction to integral domains $R$ is so that the argument of \cref{MonotoneQH} applies; we cannot simply work over $\Z$ and tensor by $R$ at the end because if $\rho$ does not arise from a class in $H^2(X; \Z^\times)$ then the maps cannot be defined over $\Z$.

\subsection{Invertibility of toric divisors}
\label{sscInvertibility}

In \cref{rmkLocJac} we mentioned that when $X$ is compact, the ring $\Lambda \llangle H \rrangle_\Delta$ (recall this was defined to be $\Lambda \otimes_{\Lambda_0} \LHD$) is a completion of the group ring $\Lambda[H_1(L; \Z)]$.  Since $H_1(L; \Z)$ is spanned by the $\nu_j$, the content of this statement is simply:

\begin{lem}
For compact toric $X$ the monomials $\lv_j = T^{\lambda_j} \tau^{\nu_j}$ are invertible in $\Lambda \llangle H \rrangle_\Delta$.
\end{lem}
\begin{proof}
Since $X$ is compact, each basic disc $u_j(z) = e^{-i\nu_j \log z} \cdot p$ extends to a holomorphic sphere.  By positivity of intersections with the toric divisors, the homology class of this sphere is of the form $m_1[u_1] + \dots + m_N[u_N]$ with $m_k - \delta_{jk} \geq 0$ for all $k$.  Since this class has zero boundary, we then have
\[
\lv_j \cdot \prod_{k=1}^N (\lv_k)^{m_k - \delta_{jk}} = T^{\sum_k m_k \lambda_k}
\]
in $\LHD$, which proves that $\lv_j$ is invertible in $\Lambda \llangle H \rrangle_\Delta$.
\end{proof}

\begin{rmk}
An alternative approach is to use \cref{CWhenDeltaCompact}.
\end{rmk}

This has the immediate corollary that any element of $\LHD$ which is congruent to some $\lv_j$ modulo $\Lambda_+$ is invertible in $\Lambda \llangle H \rrangle_\Delta$.  By the Kodaira--Spencer isomorphism we deduce:

\begin{cor}
For any compact toric variety $X$, the toric divisor classes are invertible in
\[
QH^*(X; \Lambda_0).
\]
This remains true under bulk deformations or $B$-fields.
\end{cor}

In the monotone case we can work over $R[T]$ for arbitrary $R$, and see that the $\lv_j$ are invertible in $R[\Gamma] \otimes_{R[T]} R[T^{\pm 1}]$.  We cannot deduce immediately that any element congruent to $\lv_j$ modulo $T$ is invertible (inverting an element of the form $1 + O(T)$ may involve an infinite series in $T$), but we do not need this to deduce invertibility of the toric divisor classes since they are sent to the $\lv_j$ on the nose by $\ks_\mathrm{mon}$.  We obtain:

\begin{cor}
For any compact monotone toric variety $X$, and any ring $R$, the toric divisor classes are invertible in $QH^*(X; R[T])$.  If $R$ is an integral domain and $\rho$ is a class in $H^2(X; R^\times)$ then the same result holds after deforming by $\rho$.
\end{cor}

\bibliography{ToricQHbiblio}
\bibliographystyle{utcapsor2}

\end{document}